\documentclass[letterpaper,11pt]{article}
\RequirePackage{amsmath, amsthm, amssymb, mathtools}
\RequirePackage{comment}
\RequirePackage{graphicx}
\RequirePackage{subcaption} 
\RequirePackage{fancyvrb}
\RequirePackage[dvipsnames,svgnames,table]{xcolor}
    \definecolor{linkcolor}{RGB}{0,0,128}
\RequirePackage[pdfpagelabels,plainpages=false,hypertexnames=true,colorlinks=true,linkcolor=linkcolor,anchorcolor=linkcolor,citecolor=linkcolor,filecolor=linkcolor,menucolor=linkcolor,runcolor=linkcolor,urlcolor=linkcolor,pdfborder={0 0 0}]{hyperref}
\RequirePackage[scale=0.95]{sourcecodepro}
\RequirePackage{mathpazo}

\RequirePackage{dsfont}
\RequirePackage{parskip}
\RequirePackage{tikz}
    \usetikzlibrary{positioning}
    \usetikzlibrary{decorations.pathreplacing}
\RequirePackage[backend=biber,backref=true,style=numeric,giveninits=true,natbib=true,maxbibnames=12]{biblatex}
\RequirePackage{microtype}
\RequirePackage{forest}

\usepackage{float}
\usepackage{romanbar}

\usepackage{orcidlink}

\usepackage{comment}
\usepackage{todonotes}
\usepackage{pgfplots}
\pgfplotsset{compat=1.17} 

\usepackage{algorithm}

\theoremstyle{plain}% default
\newtheorem{prototheorem}{theorem}
\newtheorem{theorem}[prototheorem]{Theorem}

\theoremstyle{plain}% default

\theoremstyle{plain}% default
\newtheorem{prototheorem3}{theorem}
\newtheorem{lemma}[prototheorem3]{Lemma}

\theoremstyle{remark}

\theoremstyle{definition}

\newcommand{\eps}{\varepsilon}

\newcommand{\W}{\boldsymbol{\mathcal{W}}}

\newcommand{\law}{\operatorname{Law}}

\newcommand{\TV}{\mathsf{TV}}

\newcommand{\rn}[1]{\Romanbar{#1}}

\newcommand{\tmix}{\tau_{\operatorname{mix}}}

\newcommand{\diag}{\operatorname{diag}}

\newcommand{\C}{\mathcal{C}}

\newcommand{\Unif}{\mathrm{Unif}}
\newcommand{\HR}{\mathrm{H\&R}}
\newcommand{\gHR}{\mathrm{gH\&R}}

\graphicspath{{Figures/}}

\title{Ballistic Convergence in Hit-and-Run Monte Carlo and a Coordinate-free Randomized Kaczmarz Algorithm}
\author{Nawaf Bou-Rabee\thanks{Department of Mathematical Sciences, Rutgers University, \href{mailto:nawaf.bourabee@rutgers.edu}  {\texttt{nawaf.bourabee@rutgers.edu}} \orcidlink{0000-0001-9280-9808}}
\and
Andreas Eberle\thanks{Institute for Applied Mathematics, University of Bonn, 
\href{mailto:eberle@uni-bonn.de}{\texttt{eberle@uni-bonn.de}} \orcidlink{0000-0003-0346-3820}}
\and
Stefan Oberd\"orster\thanks{Institute for Applied Mathematics, University of Bonn, 
\href{mailto:oberdoerster@uni-bonn.de}{\texttt{oberdoerster@uni-bonn.de}} \orcidlink{0009-0004-0734-0101}}
}

\date{December 2024}

\begin{document}

\maketitle

\begin{abstract}
Hit-and-Run is a coordinate-free Gibbs sampler, yet the quantitative advantages of its coordinate-free property remain largely unexplored beyond empirical studies.  In this paper, we prove sharp estimates for the Wasserstein contraction of Hit-and-Run in Gaussian target measures via coupling methods and conclude mixing time bounds.  Our results uncover ballistic and superdiffusive convergence rates in certain settings.  Furthermore, we extend these insights to a coordinate-free variant of the randomized Kaczmarz algorithm, an iterative method for linear systems, and demonstrate analogous convergence rates.  These findings offer new insights into the advantages and limitations of coordinate-free methods for both sampling and optimization.
%Our results demonstrate that, up to logarithmic factors, its mixing time scales as $\sqrt{\kappa} d$ in cases where low-curvature directions sufficiently overlap with the unit sphere, where $d$ is the dimension and $\kappa$ is the condition number.  
\end{abstract}

\section{Introduction}

The Gibbs sampler, also known as Glauber dynamics or the heat-bath algorithm, is a Markov chain Monte Carlo method for sampling from complex multivariate distributions by alternately sampling from  each variate’s one-dimensional conditional distribution.  This approach reduces the high-dimensional sampling problem to dimension one; for overviews,  see \cite{casella1992explaining,mcbook}.  The method was introduced by Glauber \cite{Glauber1963} for simulations of Ising models, and later introduced to statistics by Geman and Geman \cite{geman1984stochastic} in the context of image processing.  Gelfand and Smith \cite{GelfandSmith1990} further popularized it for Bayesian inference, and today, the Gibbs sampler is widely available in probabilistic programming languages such as BUGS \cite{spiegelhalter1996bugs}, JAGS \cite{plummer2003jags}, and Nimble \cite{nimble-article:2017}.   

The Gibbs sampler relies on a chosen coordinate system.  When the coordinate directions are not aligned with the geometry of the target distribution -- particularly when the components are highly correlated -- the  sampler may exhibit diffusive behavior, slowing convergence.  For example, in a centered bivariate normal with unit variances and correlation coefficient $\rho \in (-1,1)$, the systematic-scan Gibbs sampler at a current state $(x,y)$, generates a Markov chain with updates \[
X \sim \mathcal{N}(\rho y, 1-\rho^2) \;, \quad \text{and} \quad Y \sim \mathcal{N}(\rho X, 1-\rho^2 ) \;. 
\]
If we consider only the evolution of the $Y$-component and define $h = 1-\rho^2$, the discrete generator of this component converges to that of an  Ornstein-Uhlenbeck process in the high-correlation limit, specifically: for any $y \in \mathbb{R}$ and twice continuously differentiable function $f: \mathbb{R} \to \mathbb{R}$, \[
\lim_{h \searrow  0} \frac{\mathbb Ef(Y) - f(y)}{h} = - y f'(y) + f''(y) \;. 
\]  This diffusive behavior of Gibbs is illustrated in the left panel of Figure~\ref{fig:HRvsGibbs}.

\begin{figure}
\noindent\begin{minipage}{.5\textwidth}
\centering
\begin{tikzpicture}[scale=4]

\draw[->] (-0.75,0) -- (0.75,0);
\draw[->] (0,-0.75) -- (0,0.75);
\draw[rotate=45, dashed, line width=1pt] (0,0) ellipse (1 and 1/10);

\draw[gray, line width=2pt] (-0.6,-0.65) -- (-0.6,-0.55);
\draw[gray, line width=2pt] (-0.6,-0.55) -- (-0.52,-0.55);
\draw[gray, line width=2pt] (-0.52,-0.55) -- (-0.52,-0.45);
\draw[gray, line width=2pt] (-0.52,-0.45) -- (-0.52,-0.50);
\draw[gray, line width=2pt] (-0.52,-0.50) -- (-0.45,-0.50);
\draw[gray, line width=2pt] (-0.45,-0.50) -- (-0.45,-0.37);
\draw[gray, line width=2pt] (-0.45,-0.50) -- (-0.45,-0.43);
\draw[gray, line width=2pt] (-0.45,-0.43) -- (-0.4,-0.43);
\draw[gray, line width=2pt] (-0.4,-0.43) -- (-0.35,-0.43);
\draw[gray, line width=2pt] (-0.35,-0.43) -- (-0.35,-0.35);

\filldraw (-0.6,-0.65) circle (0.5pt);
\filldraw (-0.6,-0.55) circle (0.5pt);
\filldraw (-0.52,-0.55) circle (0.5pt);
\filldraw (-0.52,-0.45) circle (0.5pt);
\filldraw (-0.52,-0.50) circle (0.5pt);
\filldraw (-0.45,-0.50) circle (0.5pt);
\filldraw (-0.45,-0.37) circle (0.5pt);
\filldraw (-0.45,-0.43) circle (0.5pt);
\filldraw (-0.4,-0.43) circle (0.5pt);
\filldraw (-0.35,-0.43) circle (0.5pt);
\filldraw (-0.35,-0.35) circle (0.5pt);

\end{tikzpicture}
\end{minipage}%
\noindent\begin{minipage}{.5\textwidth}
\centering
\begin{tikzpicture}[scale=4]

\draw[->] (-0.75,0) -- (0.75,0);
\draw[->] (0,-0.75) -- (0,0.75);
\draw[rotate=45, dashed, line width=1pt] (0,0) ellipse (1 and 1/10);

\draw[gray, line width=2pt] (-0.6,-0.65) -- (-0.63,-0.6);
\draw[gray, line width=2pt] (-0.63,-0.6) -- (-0.55,-0.58);
\draw[gray, line width=2pt] (-0.55,-0.58) -- (-0.6,-0.55);
\draw[gray, line width=2pt] (-0.6,-0.55) -- (0.55,0.5);
\draw[gray, line width=2pt] (0.55,0.5) -- (0.45,0.48);
\draw[gray, line width=2pt] (0.45,0.48) -- (0.45,0.37);
\draw[gray, line width=2pt] (0.45,0.37) -- (0.27,0.30);
\draw[gray, line width=2pt] (0.27,0.30) -- (0.25,0.15);
\draw[gray, line width=2pt] (0.25,0.15) -- (-0.25,-0.15);
\draw[gray, line width=2pt] (-0.25,-0.15) -- (-0.2,-0.28);

\filldraw (-0.6,-0.65) circle (0.5pt);
\filldraw (-0.63,-0.6) circle (0.5pt);
\filldraw (-0.55,-0.58) circle (0.5pt);
\filldraw (-0.6,-0.55) circle (0.5pt);
\filldraw (0.55,0.5) circle (0.5pt);
\filldraw (0.45,0.48) circle (0.5pt);
\filldraw (0.45,0.37) circle (0.5pt);
\filldraw (0.27,0.30) circle (0.5pt);
\filldraw (0.25,0.15) circle (0.5pt);
\filldraw (-0.25,-0.15) circle (0.5pt);
\filldraw (-0.2,-0.28) circle (0.5pt);

\end{tikzpicture}
\end{minipage}%
\caption{\textit{Comparison of 10 steps of Gibbs vs.~Hit-and-Run in a narrow bivariate Gaussian target with condition number $\kappa\gg1$:   Gibbs is limited to small, incremental steps along the coordinate axes, resulting in a slow diffusive mixing rate of $\kappa^{-1}$. In contrast, Hit-and-Run samples directions uniformly and can take large, global steps whenever its sampled direction sufficiently aligns with the major axis of the ellipse.  This enables Hit-and-Run to achieve a faster ballistic mixing rate of $\kappa^{-1/2}$ (see Section \ref{sec:ballistic}).}}
\label{fig:HRvsGibbs}
\end{figure}
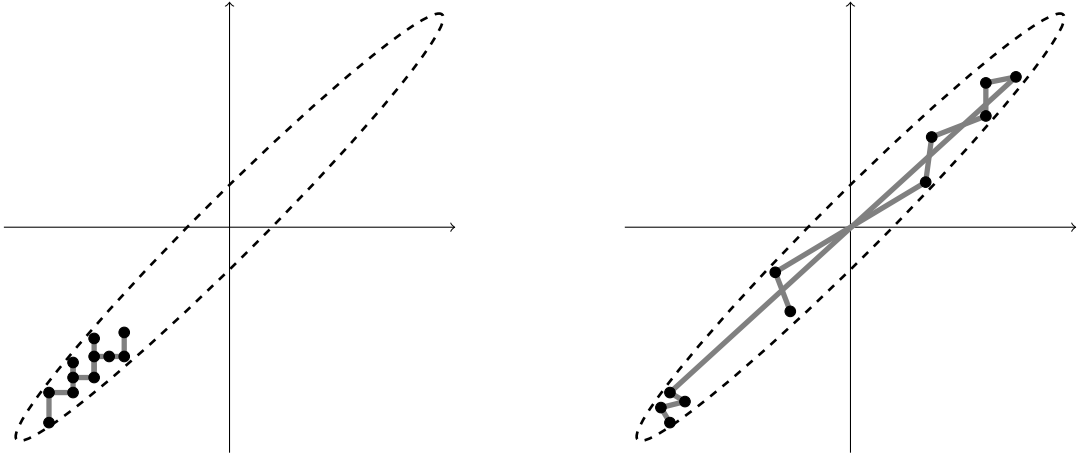

The Gibbs sampler’s reliance on a specific coordinate system motivates Hit-and-Run,  a coordinate-free alternative.  Hit-and-Run samples from the target ``conditioned’’ to a line through the current state chosen uniformly at random.  As described by Andersen and Diaconis \cite{Diaconis2007HitandRun}, the algorithm ``hits’’ a uniformly distributed direction and then ``runs’’ within that direction from the current state.  The basic Hit-and-Run method for sampling from an arbitrary probability density in high-dimensional Euclidean space was first introduced by Turchin \cite{Turchin1971}, and later by Smith \cite{Smith1984} for sampling from the uniform distribution over a convex body.  

Hit-and-Run is a well-studied method for uniform sampling over convex bodies \cite{Lovasz99,lovasz2004hit,vempala2005geometric,lovasz2006fast}.   Lovász \cite{Lovasz99} used conductance arguments to prove a sharp diffusive mixing time bound for Hit-and-Run, showing that it scales linearly with the condition number of the body, under the assumption of a warm start.
%Specifically, to approximately sample the uniform distribution from a convex body $K \subset \mathbb{R}^{d}$, a mixing time bound of $O(d^2 R^2 / r^2)$ was proven, where $R$ and $r$  are the radii of the larger circumscribed ball and the smaller inscribed ball, respectively.  The ratio $R^2 / r^2$  can be interpreted as the condition number of the target distribution.
Later, Lovász and Vempala \cite{lovasz2006fast} removed the warm start assumption by adding a warm-up phase \cite{lovasz2004hit}.  

In contrast, the ball walk \cite{kannan1997random}, while simpler to implement, can require exponential time to escape corners, making it less practical for general convex bodies.  The recently proposed In-and-Out sampler \cite{inandout2024}, a proximal sampler for uniform distributions over convex bodies, achieves similar runtime complexity to the ball walk but offers stronger guarantees in terms of Rényi divergence.  Appendix A of \cite{inandout2024} provides a detailed comparison of the ball walk, Hit-and-Run methods, and related methods for uniform sampling over convex bodies.

Chen and Eldan \cite{chen2022hit} improved upon Lovász's bounds for Hit-and-Run in isotropic bodies by using stochastic localization \cite{chen2022localization}, replacing the dependence on the condition number with a dependence on the KLS constant.  More recently,  Ascolani, Lavenant and Zanella~\cite{ascolani2024} showed relative entropy contraction at a diffusive rate for both Gibbs and Hit-and-Run in strongly log-concave distributions, a setting that generalizes uniform sampling over convex bodies. 

Hit-and-Run has inspired numerous variations and generalizations over the years \cite{ChenSchmeiser1993,Belisle1998,Diaconis2007HitandRun}. Andersen and Diaconis \cite{Diaconis2007HitandRun} introduced a significant generalization of Hit-and-Run that unifies a wide range of methods, including the Gibbs sampler, Swendsen-Wang algorithm, data augmentation, auxiliary variable Markov chains \cite{RudolfUllrich2018}, slice sampling, and the Burnside process. This framework allows the next state to be sampled from an abstract line through the current state, where the direction is drawn from a generalized distribution that need not be uniform.  Recently, this framework has been extended to include locally adaptive Hamiltonian Monte Carlo methods, such as the No-U-Turn Sampler \cite{BouRabeeCarpenterMarsden2024}. 

\paragraph{Main result for Hit-and-Run}
Unlike the Gibbs sampler, which is constrained to moving along the coordinate axes,  Hit-and-Run can move in any direction.  This key distinction allows Hit-and-Run to make \emph{global moves}, even when Gibbs is restricted to local ones. Figure \ref{fig:HRvsGibbs} illustrates this advantage, which has been empirically observed in bivariate Gaussians \cite{ChenSchmeiser1993}, but has not yet been rigorously quantified.  A key contribution of this paper is to provide a rigorous quantification of this advantage of coordinate-free Monte Carlo methods and to show its limitations.

We focus on centered Gaussian target measures with covariance matrix $\C$.   This setting is not trivial because, even for Gaussian targets, the transition distribution of Hit-and-Run is not Gaussian, see Figure \ref{fig:hr_densities}.  In this context, we state our results for the broader class of \emph{generalized Hit-and-Run samplers}, which allows directions to be chosen from an arbitrary distribution $\tau$ on the unit sphere, rather than just the uniform distribution.  

To obtain a mixing time upper bound, we use a coupling argument.  Specifically, we consider a synchronous coupling of two copies of the generalized Hit-and-Run sampler starting from  different initial points.  Remarkably, at each  step, the difference between the components of this coupling is projected onto a random $(d-1)$-dimensional linear subspace.  This insight allows us to apply a simple yet powerful contraction estimate for such random projections formalized in Lemma \ref{lem:contrproj}.  

Our analysis leads to sharp estimates on the Wasserstein contraction of the generalized Hit-and-Run sampler guaranteeing a rate given by
\[ \rho\ =\ \frac12\inf_{|\zeta|=1}\mathbb E_{v\sim\tau}\Bigr(\zeta\cdot\frac{\C^{-1/2}v}{|\C^{-1/2}v|}\Bigr)^2 \]
which, when combined with a one-step overlap bound, yields a bound on the total variation mixing time.  Importantly, our coupling-based approach does not require the initial distribution to be warm.  

Table \ref{tab:lowdimresults} summarizes the resulting rates for Hit-and-Run, i.e., $\tau=\Unif(\mathbb S^{d-1})$, in various low-dimensional cases.  Notably, our results rigorously explain the empirically observed ballistic speed-up in dimension two \cite{ChenSchmeiser1993}.  Beyond the bivariate case, we show that this speed-up occurs only if the number of high modes with respect to $\C^{-1}$ is sufficiently small relative to the number of low modes.

\begin{table}
    \centering
    \begin{tabular}{|lllc|}
    \hline
    $d$ &  \multicolumn{1}{c}{$\C$} & \multicolumn{1}{c}{$\rho\asymp$} & Speed of Mixing\\
    \hline
    2   &  $\diag(\kappa,1)$        & $\kappa^{-1/2}$            &ballistic \\
    3   &  $\diag(\kappa,\kappa,1)$ & $\kappa^{-1/2}$            &ballistic \\
    3   &  $\diag(\kappa,1,1)$      & $\kappa^{-1}\log\kappa$    & superdiffusive \\
    4   &  $\diag(\kappa,1,1,1)$    & $\kappa^{-1}$              &diffusive \\
    \hline
    \end{tabular}
    \caption{\textit{Mixing speed and Wasserstein contraction rate $\rho$ for Hit-and-Run for low-dimensional Gaussian target measures $\mathcal N(0,\C)$ with varying combinations of low ($\kappa^{-1} \ll 1$) and high modes with respect to $\C^{-1}$, see Section \ref{sec:ballistic}.  The resulting mixing time bounds are of order $\rho^{-1}$, up to logarithmic factors.  The symbol $\asymp$ indicates quantities of the same order, see \eqref{eq:asymp}.  Note that these bounds hold for any rotation of the given Gaussians, reflecting Hit-and-Run's invariance under rotations.  In particular, the first row quantifies the ballistic mixing of Hit-and-Run illustrated in Figure \ref{fig:HRvsGibbs}.}}
    \label{tab:lowdimresults}
\end{table}

\paragraph{Main result for coordinate-free randomized Kaczmarz}  
Furthermore, we extend  our findings  on coordinate-free Monte Carlo methods to analogous optimization methods by examining the randomized Kaczmarz algorithm \cite{vershynin09}.  While the analogy between Monte Carlo and optimization methods is well established in the context of Gibbs samplers \cite{roberts1997updating}, it is less well-known for Hit-and-Run.  Moreover, the advantages of coordinate-free approaches, to our knowledge, has not been rigorously quantified.  Remarkably, the general contraction lemma for random projections formalized in Lemma \ref{lem:contrproj}, naturally applies to the convergence analysis of these methods.   

The randomized Kaczmarz algorithm  is an iterative method for approximately solving overdetermined linear systems $Ax=b$.  It works by iteratively projecting onto solution hyperplanes $\{e_i\cdot(Ax-b)=0\}$ of randomly chosen equations from the system, where $e_i$ denotes the $i$-th canonical basis vector.  This algorithm is by construction restricted to the coordinate directions defined by the canonical basis of $\mathbb{R}^d$, making it analogous to random-scan Gibbs samplers.

Motivated by this restriction to particular coordinate directions, we introduce a \emph{coordinate-free randomized Kaczmarz algorithm} which projects onto solution hyperplanes \[ \{v\cdot(Ax-b)=0\} \;, \quad  v\sim\Unif(\mathbb S^{d-1}) \;,
\] which are random linear combinations of the system's equations.  Extending this further, we consider the more general class of \emph{generalized randomized Kaczmarz algorithms}, which allow for an arbitrary distribution $\tau$ on the unit sphere.  For this generalization, we prove that the mean error decays at the rate
\[ \rho\ =\ \frac12\inf_{|\zeta|=1}\mathbb E_{v\sim\tau}\Bigr(\zeta\cdot\frac{A^Tv}{|A^Tv|}\Bigr)^2\;. \]

This result follows from applying Lemma \ref{lem:contrproj}, by interpreting the error as the difference between two copies, one of which is at the solution. In each iteration, the difference is projected onto a random $(d-1)$-dimensional linear subspace, enabling the application of the general contraction lemma for random projections given below.  As a special case, our result recovers the diffusive convergence rate established in \cite{vershynin09} for the coordinate-bound variant.  Moreover, the ideas developed for Hit-and-Run transfer seamlessly to the randomized Kaczmarz algorithm, revealing the potential for a diffusive-to-ballistic speed-up in its coordinate-free variant.
%This connection underscores the broader applicability of our insights and highlights the advantages of coordinate-free approaches in both Monte Carlo and optimization methods.

\paragraph{General contraction lemma for random projections}
We now state and prove a central lemma that quantifies the average contraction resulting from random projections onto $(d-1)$-dimensional linear subspaces of $\mathbb R^d$.  This lemma plays a crucial role in the convergence analyses of both Hit-and-Run and coordinate-free randomized Kaczmarz, as illustrated in Figure \ref{fig:useofprojection}.

\begin{figure}[t]
\noindent\begin{minipage}[b]{.48\textwidth}
\flushleft
\begin{tikzpicture}[scale=1.3]
\clip (-3,-2.5) rectangle (3,2.5);
%\draw (-3,-2.5) rectangle (3,2.5);
\draw[line width=1pt] (-2,2) -- (2,-2) node[pos=0.1, above right] {$\{v\cdot\C^{-1/2}x=0\}$};
\draw[line width=1pt,dashed] (0,0) -- (-1.5,-1.5);
\filldraw (0,0) circle (2.5pt) node[above right] {$\C^{-1/2}(X_{k+1}-\tilde X_{k+1})$};
\filldraw (-1.5,-1.5) circle (2.5pt) node[below] {$\C^{-1/2}(X_k-\tilde X_k)$};
\end{tikzpicture}
\caption*{\textit{(a) Difference between components of a synchronous coupling of two copies of Hit-and-Run Monte Carlo. % in natural coordinates.% corresponds to projections on random $(d-1)$-dimensional linear subspaces with $v\sim\tau$.
}}
\end{minipage}%
\hfill
\noindent\begin{minipage}[b]{.48\textwidth}
\centering
\begin{tikzpicture}[scale=1.3]
\clip (-3,-2.5) rectangle (3,2.5);
%\draw (-3,-2.5) rectangle (3,2.5);
\draw[line width=1pt] (-2,2) -- (2,-2) node[pos=0.1, above right] {$\{v\cdot Ax=0\}$};
\draw[line width=1pt,dashed] (0,0) -- (-1.5,-1.5);
\filldraw (0,0) circle (2.5pt) node[above right] {$x_{k+1}-x^\ast$};
\filldraw (-1.5,-1.5) circle (2.5pt) node[below] {$x_k-x^\ast$};
\end{tikzpicture}
\caption*{\textit{(b) Difference between coordinate-free randomized Kaczmarz and the solution $x^\ast$ of the linear system.
%corresponds to projections on random $(d-1)$-dimensional linear subspaces with $v\sim\tau$.
}}
\end{minipage}%
\caption{\textit{Random projections onto hyperplanes (solid lines) defined using $v\sim\Unif(\mathbb S^{d-1})$.
}}
\label{fig:useofprojection}
\end{figure}

\medskip

\begin{lemma}\label{lem:contrproj}
Let $\eta$ be a probability measure on $\mathbb R^d$.  For $w\sim\eta$, define the random projection
\[ \Pi_w=I_d-\frac{w\otimes w}{|w|^2} \]
which projects any vector onto the orthogonal complement of $\operatorname{span}(w)$.  For all $z\in\mathbb R^d$, $\Pi_w$ satisfies
\begin{equation}\label{eq:generalrate}
    \mathbb E_{w\sim\eta}|\Pi_w\,z|^2\ \leq\ \Bigr(1-\inf_{|\zeta|=1}\mathbb E_{w\sim\eta}\Bigr(\zeta\cdot\frac{w}{|w|}\Bigr)^2\Bigr)\,|z|^2 \;. 
\end{equation}
\end{lemma}

The expectation on the right-hand side of \eqref{eq:generalrate} measures the average contraction rate of a fixed vector $z$ under the random projection $\Pi_w$ for $w\sim\eta$.  Taking the infimum over all unit vectors $\zeta$ yields a global rate of average contraction for the random projection.

\begin{proof}[Proof of Lemma \ref{lem:contrproj}]
Let $z\in\mathbb R^d$.
For all unit vectors $w \in \mathbb{R}^d$, it holds
\[ |\Pi_w\,z|^2\ =\ \bigr|z-(z\cdot w)w\bigr|^2\ =\ \Bigr(1\ -\ \Bigr(\frac{z}{|z|}\cdot w\Bigr)^2\,\Bigr)\,|z|^2\;. \]
Using this, we can compute
\begin{align*}
    \mathbb E_{w\sim\eta}|\Pi_w\,z|^2\ &=\ \mathbb E_{w\sim\eta}|\Pi_{w/|w|}\,z|^2\ \leq\ \Bigr(1-\inf_{|\zeta|=1}\mathbb E_{w\sim\eta}\Bigr(\zeta\cdot\frac{w}{|w|}\Bigr)^2\,\Bigr)\,|z|^2\;.
\end{align*}
This completes the proof. \end{proof}

%Our proof relies on a coupling of two copies of Hit-and-Run after $n$ steps. For the first $n-1$ step, we use a synchronous coupling to bring the copies closer, followed by a one-shot coupling at step $n$ that coalesces the copies with high probability.   Such couplings are standard in the literature \cite{roberts2002one,madras2010quantitative,EbMa2019,monmarche2021high, BouRabeeEberle2023,BouRabeeOberdoerster2024}.  By the coupling lemma, this approach provides an upper bound on the total variation distance between a copy starting from an arbitrary distribution and one starting from the target distribution.   As we will show, the synchronous coupling contracts due to random projections at each transition step. 

%Since the velocity variable in the No-U-Turn Sampler (NUTS) functions similarly to the direction in Hit-and-Run, we conjecture that NUTS will exhibit a similar dependence on the condition number.  However, because NUTS ``runs’’ along this direction using leapfrog approximations of Hamiltonian flows, it scales more favorably with dimension than Hit-and-Run, as quantified in \cite{bou2024mixing}.  The leapfrog approximation to the Hamiltonian flow allows NUTS to efficiently explore specific regions of state space where the target measure is concentrated in high dimension.

\paragraph{Organization of paper}
The remainder of the paper is organized as follows.  Section~\ref{sec:HR} introduces a generalized Hit-and-Run sampler, encompassing Hit-and-Run as a special case.  Sections~\ref{sec:wasserlower} and~\ref{sec:wasserupper} provide upper and lower bounds, respectively, on the Wasserstein contraction rate for this class of methods when applied to Gaussian target measures.  Section~\ref{sec:ballistic} discusses the rates achieved by Hit-and-Run.  Section~\ref{sec:mixing} combines the Wasserstein contraction results with a one-step overlap bound from Section~\ref{sec:overlap} to obtain mixing time upper bounds for Hit-and-Run.  Finally, Section~\ref{sec:kaczmarz} extends these  techniques to analyze the generalized randomized Kaczmarz algorithm.   

%In \cite{Lovasz1999}, in the context of sampling from a convex body $K \subset \mathbb{R}^{d}$, a mixing time bound of $O(d^2 R^2 / r^2)$ is proved for Hit-and-Run where $R$ is the radius of the larger circumscribed ball and $r$ is the radius of the smaller inscribed ball.  This bound is claimed to be sharp.  This condition number dependence is consistent with the very recent finding of \cite[Theorem 8.1]{ascolani2024}, which is stated in a general strongly log-concave setting. However, they find a better dimension dependence of $O(d)$.

%In \cite{ChenSchmeiser1993}, Hit-and-Run is empirically compared with Gibbs in the context of a bivariate Gaussian.  The authors find that Hit-and-Run generally outperforms Gibbs, except when the correlation is zero, where Gibbs provides independent samples and performs better. As the correlation increases, Hit-and-Run's performance improves, significantly outperforming Gibbs for high correlations. 

%\subsection*{Open-source code and reproducibility}

\subsection*{Acknowledgements}

We wish to acknowledge Ahmed Bou-Rabee, Bob Carpenter, Francis L\"orler, and Andre Wibosono for useful discussions.
N.~Bou-Rabee was partially supported by NSF grant  No.~DMS-2111224.
A.~Eberle and S.~Oberd\"orster were supported by the Deutsche Forschungsgemeinschaft (DFG, German Research Foundation) under Germany’s Excellence Strategy – EXC-2047/1 – 390685813.

\section{The generalized Hit-and-Run sampler}\label{sec:HR}

In this section, we present a rigorous definition of the Hit-and-Run transition kernel and subsequently extend it by allowing for arbitrary distributions of the directions.

Let $\mu$ be an absolutely continuous target probability measure on $\mathbb R^d$ with density with respect to Lebesgue measure also denoted by $\mu(x)$, $x\in\mathbb R^d$.  To describe the Hit-and-Run transition step rigorously, let $\mathbb S^{d-1}$ denote the unit sphere in $\mathbb{R}^d$, and define the line through $x \in \mathbb{R}^d$ in the direction $v \in \mathbb S^{d-1}$ as
\[ L(x,v) = \{  x + h \, v \mid h \in \mathbb{R} \} \;. \]
It corresponds to the image of the displacement map $\Delta_{x,v}:\mathbb R\to\mathbb R^d$, $\Delta_{x,v}(h)=x+hv$ which maps a scalar displacement $h$ to the corresponding point on the line.
Once a parametrization of the line is fixed, a regular version of the conditional distribution of $\mu$ given $L(x,v)$ is the push forward
\[ \mu(\,\cdot\,|L(x,v))\ =\ \mu_{x,v}\circ\Delta_{x,v}^{-1} \]
of the one-dimensional displacement measure $\mu_{x,v}$ on $\mathbb R$ defined via its almost everywhere finite density
\begin{equation}\label{eq:sdm}
    \mu_{x,v}(h)\ \propto\ \mu(x+hv) \;.
\end{equation}

A transition of Hit-and-Run proceeds as follows: Given the current state $x \in \mathbb{R}^d$, the direction $v\sim\Unif(\mathbb S^{d-1})$ defines a line through $x$ drawn uniformly at random.  The next state $X\sim\mu(\,\cdot\,|L(x,v))$ is realized as $X=\Delta_{x,v}(H)=x+Hv$ with a random displacement $H\sim\mu_{x,v}$.  See Figure \ref{fig:HR_transition_step}.

Note that the line $L(x,v)$ and correspondingly the distribution $\mu(\cdot|L(x,v))$ do not depend on the specific choice of anchor point $x$ on the line while the displacement map $\Delta_{x,v}$ and measure $\mu_{x,v}$ do.

\begin{algorithm}[H] 
\caption{Hit-and-Run $X\sim\pi_{\HR}(x,\cdot)$}\label{algo:HR}
\begin{description}
\item[(Hit Step)] Sample a direction $v \sim \Unif(\mathbb S^{d-1})$.
\item[(Run Step)] Sample a scalar displacement $H \sim \mu_{x,v}$.
\item[(Output)] Return $X = x + H \, v$.
\end{description}
\end{algorithm}

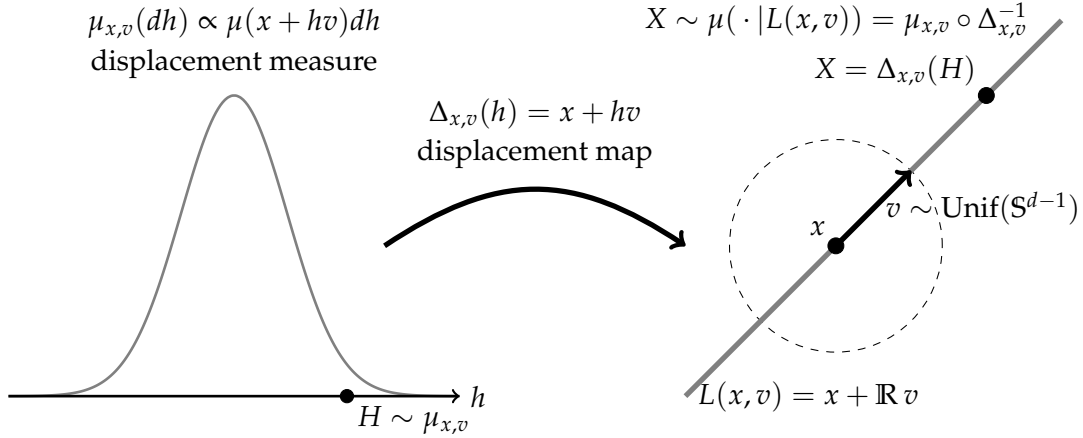
\begin{figure}[t]
\centering
\begin{tikzpicture}

\clip (0,0.5) rectangle (16,7);
%\draw (0,0) rectangle (16,8);
%\draw[thick] (8,0) -- (8,8);
%\draw[step=1.0,ultra thin] (0,0) grid (16,8);

\draw [gray,line width=1pt,smooth,samples=100,domain=0:6] plot({\x+1},{4*e^(-(\x-3)^2)+1});
\draw[line width=1pt,->] (1,1) -- (7,1) node[at end, right]{$h$};
\node[align=center] at (4,5.7) {$\mu_{x,v}(dh)\propto\mu(x+hv)dh$\\displacement measure};
\filldraw (5.5,1) circle (2.5pt) node[below right]{$H\sim\mu_{x,v}$};

\draw[dashed] (12,3) circle ({sqrt(2)});
\draw[gray, line width=2pt] (10,1) -- (15,6) node[black,at start, right] {$L(x,v)=x+\mathbb R\,v$};
\filldraw (12,3) circle (3pt) node[above left]{$x$};
\draw[line width=2pt, ->] (12, 3) -- (13, 4) node[pos=0.5, right] {$v\sim\Unif(\mathbb{S}^{d-1})$};
\filldraw (14,5) circle (3pt) node[above left]{$X=\Delta_{x,v}(H)$};
\node[align=center] at (12,6) {$X\sim\mu(\,\cdot\,|L(x,v))=\mu_{x,v}\circ\Delta_{x,v}^{-1}$};

\draw[line width=2pt, ->] (6,3) .. controls (7.5,4) and (8.5,4) .. (10,3);
\node[align=center] at (8,4.5) {$\Delta_{x,v}(h)=x+hv$\\displacement map};

%\draw[-, dashed, thick] (-3, -1.5) -- (5, 2.5);
%\filldraw[black] (0, 0) circle (3pt) node[below right] {$x \in \mathbb{R}^d$};
%\filldraw[black] (4, 2) circle (3pt) node[below right] {$X\sim\mu(\,\cdot\,|L(x,v))$};
\end{tikzpicture}

\caption{\textit{
Transition step of Hit-and-Run from current state $x$ to next state $X$.
}}
\label{fig:HR_transition_step}
\end{figure}

Let $\sigma_{d-1}=\Unif(\mathbb S^{d-1})$ be the normalized Haar measure on $\mathbb S^{d-1}$.  The transition kernel of Hit-and-Run is given by 
\begin{equation}\label{eq:H&Rkernel}
\pi_{\HR}(x,A)\ =\ \int_{\mathbb{S}^{d-1}}\int_{\mathbb{R}} 1_{A} \left(x + h v \right) \, \mu_{x,v}(dh) \, \sigma_{d-1}(dv)   
\end{equation}
for $x\in\mathbb R^d$ and measurable sets $A \subset \mathbb{R}^d$.
Changing from spherical to Cartesian coordinates yields the equivalent representation
\begin{equation}\label{eq:HRcart}
\pi_{\HR}(x,dy) \ = \  1_{ \{ y \ne x \} } \, \frac{2}{ a_{d-1}}  \, \frac{1}{\left|y-x \right|^{d-1}} \, \mu_{x,\frac{y-x}{|y-x|}}\bigr(|y-x|\bigr)  \,  dy
\end{equation}
where $a_{d-1}$ denotes the surface area of $\mathbb{S}^{d-1}$.  Similar representations appear in \cite{ChenSchmeiser1993,chen2022hit}.  This representation will be used to quantify the overlap between the transition distributions of two copies of Hit-and-Run starting at different points, see Lemma \ref{lem:overlap1}.

By allowing for more general probability distributions $\tau$ on $\mathbb S^{d-1}$ in the \emph{Hit Step}, we can define the generalized Hit-and-Run sampler with transition kernel
\begin{equation}\label{eq:gGibbskernel}
\pi_{\mathrm{\gHR}}(x,A)\  =\ \int_{\mathbb{S}^{d-1}}\int_{\mathbb{R}} 1_{A} \left(x + h v \right) \, \mu_{x,v}(dh) \, \tau(dv)\;.
\end{equation}
This generalized Hit-and-Run sampler encompasses not only Hit-and-Run for $\tau = \Unif(\mathbb S^{d-1})$, but also the random-scan Gibbs sampler corresponding to $\tau=\frac1d\sum_{i=1}^d\delta_{e_i}$ where $\{e_i\}_{1 \le i \le d}$ is the canonical basis of $\mathbb R^d$.  In the latter case, $\tau$ equals the uniform distribution over the $d$ coordinate directions.  The next theorem ensures reversibility of this general class of samplers with respect to the target distribution.

\medskip

\begin{theorem}\label{thm:rev}
The transition kernel $\pi_{\gHR}$ of generalized Hit-and-Run is reversible with respect to $\mu$. 
\end{theorem}

\begin{proof}
Since $\mu_{x,v}(dh)=\frac{\mu(x+hv)}{\int_{\mathbb R}\mu(x+sv)ds}dh$ for all $x,v\in\mathbb R^d$ by \eqref{eq:sdm}, inserting \eqref{eq:H&Rkernel} and changing variables twice shows for any measurable sets $A, B \subset \mathbb{R}^d$, \begin{align*}
  \int_A  \pi_{\gHR}(x,B)  \mu(dx)
  &=\int_{\mathbb{R}^d}\int_{\mathbb{S}^{d-1}}\int_{\mathbb{R}}  1_{A}(x) \, 1_B(x + h v) \, \mu_{x,v}(dh) \, \tau(dv) \, \mu(dx)  \\
  &=  \int_{\mathbb{R}^d}\int_{\mathbb{S}^{d-1}}\int_{\mathbb{R}}1_{A}(x) \,  1_B(x + h v) \,  \frac{\mu(x) \,  \mu(x+ h v )} {\int_{\mathbb{R}} \mu(x+s v) ds } \, dh \, \tau(dv) \, dx \\ 
  &=  \int_{\mathbb{R}^d}\int_{\mathbb{S}^{d-1}}\int_{\mathbb{R}}  1_{A}(x - h v) \,  1_B(x) \, \frac{\mu(x-hv) \,  \mu(x)} {\int_{\mathbb{R}} \mu(x+s v) ds } \, dh \, \tau(dv) \, dx \\
  &=  \int_{\mathbb{R}^d}\int_{\mathbb{S}^{d-1}}\int_{\mathbb{R}}  1_{A}(x + h v) \,  1_B(x) \, \frac{\mu(x+hv) \,  \mu(x)} {\int_{\mathbb{R}} \mu(x+s v) ds } \, dh \, \tau(dv) \, dx \\
  &=  \int_{\mathbb{R}^d}\int_{\mathbb{S}^{d-1}}\int_{\mathbb{R}} 1_{A}(x + h v) \,  1_B(x) \, \mu_{x,v}(dh) \, \tau(dv) \, \mu(dx)  \\
  & =   \int_B \pi_{\gHR}(x,A) \mu(dx) \;.
\end{align*}    
Thus, the kernel $\pi_{\gHR}$  is reversible with respect to $\mu$.
\end{proof}

\section{Generalized Hit-and-Run in Gaussian distributions}\label{sec:gGibbsinGaussians}

Consider the centered Gaussian target measure
\[ \gamma^{\C}\ =\ \mathcal{N}(0,\mathcal{C})\quad\text{on $\mathbb R^d$.} \]
Given a line $L(x,v)$, the corresponding scalar displacement measure, as defined in \eqref{eq:sdm}, is characterized by its density
\[ \begin{aligned}[t]
\gamma^\C_{x,v}(h)\ \propto\ \gamma^\C(x+hv)\ \propto\ e^{-\frac12|\C^{-1/2}v|^2\,h^2-x\cdot\C^{-1} v\,h}\ \propto\ \exp\Bigr(-\frac{|\C^{-1/2}v|^2}{2}\Bigr(h+\frac{x\cdot\C^{-1}v}{|\C^{-1/2}v|^2}\Bigr)^2\Bigr)\;.
\end{aligned} \]
%\[ \begin{aligned}[t] \gamma^\C_{x,v}(h)\ =\ \frac{\gamma_\C(x+hv)}{\int_{\mathbb R}\gamma_\C(x+sv)ds}\ =\ \frac{e^{-\frac12|\C^{-1/2}v|^2\,h^2-x\cdot\C^{-1} v\,h}}{\int_{\mathbb R}e^{-\frac12|\C^{-1/2}v|^2\,s^2-x\cdot\C^{-1} v\,s}ds}\ =\ \frac{e^{-\frac{|\C^{-1/2}v|^2}{2}\bigr(h+\frac{x\cdot\C^{-1}v}{|\C^{-1/2}v|^2}\bigr)^2}}{\sqrt{2\pi}|\C^{-1/2}v|^{-1}}\;. \end{aligned} \]

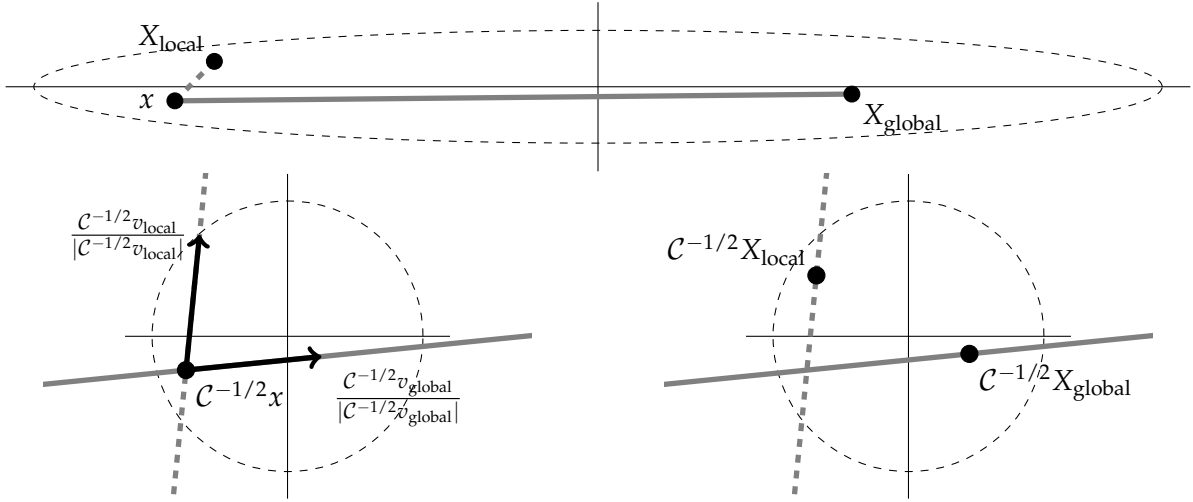
\begin{figure}[t]
\centering
\begin{tikzpicture}[scale=.75]
%\draw (-10.5,-1.5) rectangle (10.5,1.5);
\clip (-10.5,-1.5) rectangle (10.5,1.5);
\draw (-10.5,0) -- (10.5,0);
\draw (0,-1.5) -- (0,1.5);
\draw[dashed] (0,0) ellipse (10 and 1);
\draw[gray,dashed,line width=2pt] (-7.5,-0.25) -- ({-7.5+0.7},{-0.25+0.7});
\filldraw ({-7.5+0.7},{-0.25+0.7}) circle (4pt) node[above left]{$X_{\mathrm{local}}$};
\draw[gray,line width=2pt] (-7.5,-0.25) -- ({-7.5+12},{-0.25+1.2/10});
\filldraw ({-7.5+12},{-0.25+1.2/10}) circle (4pt) node[below right]{$X_{\mathrm{global}}$};
\filldraw[black] (-7.5,-0.25) circle (4pt);
\node at (-8,-0.25) {$x$};
\end{tikzpicture}
%\noindent\begin{minipage}{.3\textwidth}
%\centering
%\begin{tikzpicture}[scale=1.6]
%\draw (-1.5,-1.4) rectangle (1.5,1.4);
%\clip (-1.5,-1.4) rectangle (1.5,1.4);
%\draw (-1.2,0) -- (1.2,0);
%\draw (0,-1.2) -- (0,1.2);
%\draw[dashed] (0,0) ellipse (1 and 1);
%\filldraw[black] (-.75,-0.25) circle (2.5pt) node[anchor=north]{$\C^{-1/2}x$};
%\end{tikzpicture}
%\end{minipage}%
\noindent\begin{minipage}{.5\textwidth}
\centering
\begin{tikzpicture}[scale=1.8]
%\draw (-1.8,-1.2) rectangle (1.8,1.2);
\clip (-1.8,-1.2) rectangle (1.8,1.2);
\draw (-1.2,0) -- (1.2,0);
\draw (0,-1.2) -- (0,1.2);
\draw[dashed] (0,0) ellipse (1 and 1);
%\draw[gray] (-.75,-0.25) circle (1);
%\foreach \k in {0,...,23} {\filldraw[gray] ({-.75+(cos((\k/24)*2*pi r)/10)/sqrt(cos((\k/24)*2*pi r)^2/100 + sin((\k/24)*2*pi r)^2)},{-.25+sin((\k/24)*2*pi r)/sqrt(cos((\k/24)*2*pi r)^2/100 + sin((\k/24)*2*pi r)^2)}) circle (2pt);}
\draw [gray, dashed, line width=2pt] plot[variable=\t,domain=-1.5:1.5,smooth,thick] ({-.75+\t/10},{-0.25+\t});
\draw [gray, line width=2pt] plot[variable=\t,domain=-1.2:3,smooth,thick] ({-.75+\t},{-0.25+\t/10});
\filldraw[black] (-.75,-0.25) circle (1.8pt) node[below right] {$\C^{-1/2}x$};
\draw[line width=2pt,->] (-.75,-0.25) -- ({-.75+1/10},{-0.25+1}) node[left] {$\frac{\C^{-1/2}v_{\mathrm{local}}}{|\C^{-1/2}v_{\mathrm{local}}|}$};
\draw[line width=2pt,->] (-.75,-0.25) -- ({-.75+1},{-0.25+1/10}) node[below right] {$\frac{\C^{-1/2}v_{\mathrm{global}}}{|\C^{-1/2}v_{\mathrm{global}}|}$};
\end{tikzpicture}
\end{minipage}%
\noindent\begin{minipage}{.5\textwidth}
\centering
\begin{tikzpicture}[scale=1.8]
%\draw (-1.8,-1.2) rectangle (1.8,1.2);
\clip (-1.8,-1.2) rectangle (1.8,1.2);
\draw (-1.2,0) -- (1.2,0);
\draw (0,-1.2) -- (0,1.2);
\draw[dashed] (0,0) ellipse (1 and 1);
\draw [gray, dashed, line width=2pt] plot[variable=\t,domain=-1.5:1.5,smooth,thick] ({-.75+\t/10},{-0.25+\t});
\draw [gray, line width=2pt] plot[variable=\t,domain=-1.2:3,smooth,thick] ({-.75+\t},{-0.25+\t/10});
\filldraw ({-0.75+0.7/10},{-0.25+0.7}) circle (1.8pt) node[above left]{$\C^{-1/2}X_{\mathrm{local}}$};
\filldraw ({-0.75+1.2},{-0.25+1.2/10}) circle (1.8pt) node[below right]{$\C^{-1/2}X_{\mathrm{global}}$};
\end{tikzpicture}
\end{minipage}%
%\begin{tikzpicture}[scale=.5]
%\draw (-10.5,0) -- (10.5,0);
%\draw (0,-1.5) -- (0,1.5);
%\draw[dashed] (0,0) ellipse (10 and 1);
%\filldraw[black] (-7.5,-0.25) circle (5pt);
%\end{tikzpicture}
\caption{\textit{Two realizations of a Hit-and-Run transition starting from $x$ with $\C=\diag(\kappa,1)$, $\kappa>1$, see \eqref{eq:HRstep}.  The directions $v_{\mathrm{local}}$ and $v_{\mathrm{global}}$ result in vastly different moves.
The distribution $\mathrm{Law}(\C^{-1/2}v/|\C^{-1/2}v|)$, for $v\sim\Unif(\mathbb S^1)$, favors directions that lead to local moves.  However, the probability of making a global move is of order $\kappa^{-1/2}$, see Figure \ref{fig:wlaw}, allowing Hit-and-Run to achieve ballistic mixing, as discussed in Section \ref{sec:ballistic}.}}
\label{fig:Gaussiantrans}
\end{figure}

Thus,
\begin{equation}\label{eq:sdmGaussian}
    \gamma^\C_{x,v}\ =\ \mathcal N\Bigr(-\frac{x\cdot\C^{-1}v}{|\C^{-1/2}v|^2},\,|\C^{-1/2}v|^{-2}\Bigr)\;.
\end{equation}
For $Z \sim \mathcal{N}(0,1)$, the random variable   \begin{equation} \label{eq:rvH}
H\ =\ - \frac{x\cdot\C^{-1}v}{|C^{-1/2} v|^2} + \frac{Z}{|\mathcal{C}^{-1/2} v|}\quad\text{satisfies}\quad H\ \sim\ \gamma^\C_{x,v}\;.
\end{equation}
With $v\sim\tau$ independent of $Z$, a transition step $X\sim\pi_{\gHR}(x,\cdot)$ of generalized Hit-and-Run with invariant measure $\gamma^\C$ satisfies
\begin{equation}\label{eq:HRstep}
\begin{aligned}
\C^{-1/2}X\ &=\ \C^{-1/2}(x+Hv)\ =\ \Bigr(I_d-\frac{\C^{-1/2}v\otimes\C^{-1/2}v}{|\C^{-1/2}v|^2}\Bigr)\C^{-1/2}x\ +\ Z\frac{\C^{-1/2}v}{|\C^{-1/2}v|}\\ 
&=\ \Pi_{\C^{-1/2}v}\,\C^{-1/2}x\ +\ Z\frac{\C^{-1/2}v}{|\C^{-1/2}v|}
\end{aligned}
\end{equation}
where  $\Pi_w=I_d-|w|^{-2}w\otimes w$, $w\in\mathbb R^d$ is the projection onto the orthogonal complement of $\operatorname{span}(w)$, as defined in Lemma \ref{lem:contrproj}. 

In the natural coordinates $\C^{-1/2}x$, the target measure turns into the canonical Gaussian measure.  Generalized Hit-and-Run then replaces the component of $\C^{-1/2}x$ in the random direction $\C^{-1/2}v/|\C^{-1/2}v|$ by a sample from the canonical Gaussian.  Note that even in the case of Hit-and-Run, the distribution of $\C^{-1/2}v/|\C^{-1/2}v|$ is not uniform, unlike the distribution of $v$.  This is illustrated in Figure \ref{fig:wlaw}.  The transition via the natural coordinates is depicted in Figure \ref{fig:Gaussiantrans}.

\begin{figure}[t]
\centering
\noindent\begin{minipage}[b]{.3\textwidth}
\centering
\begin{tikzpicture}[scale=1.7]
%\draw[PineGreen,line width=6pt,smooth,samples=10,domain=-1:1] plot ({cos((\x/24)*2*pi r)},{sin((\x/24)*2*pi r)});
%\draw[PineGreen,line width=6pt,smooth,samples=10,domain=11:13] plot ({cos((\x/24)*2*pi r)},{sin((\x/24)*2*pi r)});
\draw (-1.3,0) -- (1.3,0);
\draw (0,-1.3) -- (0,1.3);
\node at (0.15,0.2) {$\alpha$};

\draw (0.25,0.2) -- (0.4,0.05);
\draw[smooth,samples=10,domain=0:1] plot ({0.5*cos((\x/24)*2*pi r)},{0.5*sin((\x/24)*2*pi r)});
\draw[line width=1pt] (0,0) circle (1);
\draw ({cos((1/24)*2*pi r)},{sin((1/24)*2*pi r)}) -- (0,0);%({-cos((1/24)*2*pi r)},{-sin((1/24)*2*pi r)});
%\draw ({cos((23/24)*2*pi r)},{sin((23/24)*2*pi r)}) -- ({-cos((23/24)*2*pi r)},{-sin((23/24)*2*pi r)});
\foreach \k in {0,...,23} {\filldraw[gray] ({cos((\k/24)*2*pi r)},{sin((\k/24)*2*pi r)}) circle (2pt);}
\filldraw[black] ({cos((1/24)*2*pi r)},{sin((1/24)*2*pi r)}) circle (2pt);
%\draw [PineGreen, line width=3pt] plot[variable=\a,domain=-0.1:0.1,smooth,thick] ({cos(\a r)},{sin(\a r)});
%\draw [PineGreen, line width=3pt] plot[variable=\a,domain=-0.1:0.1,smooth,thick] ({cos((\a+pi) r)},{sin((\a+pi) r)});
\end{tikzpicture}
\caption*{$\law(v)=\Unif(\mathbb S^1)$}
\end{minipage}%
\noindent\begin{minipage}[b]{.3\textwidth}
\centering
\begin{tikzpicture}[scale=1.7]
%\draw[PineGreen,line width=6pt,smooth,samples=10,domain=-1:1] plot ({cos((\x/24)*2*pi r)/5},{sin((\x/24)*2*pi r)});
%\draw[PineGreen,line width=6pt,smooth,samples=10,domain=11:13] plot ({cos((\x/24)*2*pi r)/5},{sin((\x/24)*2*pi r)});
\draw (-1.3,0) -- (1.3,0);
\draw (0,-1.3) -- (0,1.3);
\draw[line width=1pt] (0,0) ellipse (1/5 and 1);
\draw ({cos((1/24)*2*pi r)/5},{sin((1/24)*2*pi r)}) -- (0,0);%({-cos((1/24)*2*pi r)/5},{-sin((1/24)*2*pi r)});
%\draw ({cos((23/24)*2*pi r)/5},{sin((23/24)*2*pi r)}) -- ({-cos((23/24)*2*pi r)/5},{-sin((23/24)*2*pi r)});
\foreach \k in {0,...,23} {\filldraw[gray] ({cos((\k/24)*2*pi r)/5},{sin((\k/24)*2*pi r)}) circle (2pt);}
\filldraw[black] ({cos((1/24)*2*pi r)/5},{sin((1/24)*2*pi r)}) circle (2pt);
%\draw [PineGreen, line width=3pt] plot[variable=\a,domain=-0.1:0.1,smooth,thick] ({cos(\a r)/5},{sin(\a r)});
%\draw [PineGreen, line width=3pt] plot[variable=\a,domain=-0.1:0.1,smooth,thick] ({cos((\a+pi) r)/5},{sin((\a+pi) r)});
\end{tikzpicture}
\caption*{$\law(\C^{-1/2}v)$}
\end{minipage}%
\noindent\begin{minipage}[b]{.3\textwidth}
\centering
\begin{tikzpicture}[scale=1.7]
%\draw[PineGreen,line width=3pt,smooth,samples=10,domain=-1:1] plot ({(cos((\x/24)*2*pi r)/5)/sqrt(cos((\x/24)*2*pi r)^2/25 + sin((\x/24)*2*pi r)^2)},{sin((\x/24)*2*pi r)/sqrt(cos((\x/24)*2*pi r)^2/25 + sin((\x/24)*2*pi r)^2)});
%\draw[PineGreen,line width=3pt,smooth,samples=10,domain=11:13] plot ({(cos((\x/24)*2*pi r)/5)/sqrt(cos((\x/24)*2*pi r)^2/25 + sin((\x/24)*2*pi r)^2)},{sin((\x/24)*2*pi r)/sqrt(cos((\x/24)*2*pi r)^2/25 + sin((\x/24)*2*pi r)^2)});
\draw (-1.3,0) -- (1.3,0);
\draw (0,-1.3) -- (0,1.3);
\draw[line width=1pt] (0,0) circle (1);
\node[right] at (0.15,0.15) {$\phi$};
\draw[smooth,samples=10,domain=0:1] plot ({0.5*((cos((\x/24)*2*pi r)/5)/sqrt(cos((\x/24)*2*pi r)^2/25 + sin((\x/24)*2*pi r)^2))},{0.5*(sin((\x/24)*2*pi r)/sqrt(cos((\x/24)*2*pi r)^2/25 + sin((\x/24)*2*pi r)^2))});
\draw ({(cos((1/24)*2*pi r)/5)/sqrt(cos((1/24)*2*pi r)^2/25 + sin((1/24)*2*pi r)^2)},{sin((1/24)*2*pi r)/sqrt(cos((1/24)*2*pi r)^2/25 + sin((1/24)*2*pi r)^2)}) -- (0,0);%({-((cos((1/24)*2*pi r)/5)/sqrt(cos((1/24)*2*pi r)^2/25 + sin((1/24)*2*pi r)^2))},{-(sin((1/24)*2*pi r)/sqrt(cos((1/24)*2*pi r)^2/25 + sin((1/24)*2*pi r)^2))});
%\draw ({-((cos((23/24)*2*pi r)/5)/sqrt(cos((23/24)*2*pi r)^2/25 + sin((23/24)*2*pi r)^2))},{-(sin((23/24)*2*pi r)/sqrt(cos((23/24)*2*pi r)^2/25 + sin((23/24)*2*pi r)^2))}) -- ({(cos((23/24)*2*pi r)/5)/sqrt(cos((23/24)*2*pi r)^2/25 + sin((23/24)*2*pi r)^2)},{sin((23/24)*2*pi r)/sqrt(cos((23/24)*2*pi r)^2/25 + sin((23/24)*2*pi r)^2)});
\foreach \k in {0,...,23} {\filldraw[gray] ({(cos((\k/24)*2*pi r)/5)/sqrt(cos((\k/24)*2*pi r)^2/25 + sin((\k/24)*2*pi r)^2)},{sin((\k/24)*2*pi r)/sqrt(cos((\k/24)*2*pi r)^2/25 + sin((\k/24)*2*pi r)^2)}) circle (2pt);}
\filldraw[black] ({(cos((1/24)*2*pi r)/5)/sqrt(cos((1/24)*2*pi r)^2/25 + sin((1/24)*2*pi r)^2)},{sin((1/24)*2*pi r)/sqrt(cos((1/24)*2*pi r)^2/25 + sin((1/24)*2*pi r)^2)}) circle (2pt);
%\draw [PineGreen, line width=3pt] plot[variable=\a,domain=-0.1:0.1,smooth,thick] ({(cos(\a r)/5)/sqrt(cos(\a r)^2/25 + sin(\a r)^2)},{sin(\a r)/sqrt(cos(\a r)^2/25 + sin(\a r)^2)});
%\draw [PineGreen, line width=3pt] plot[variable=\a,domain=-0.1:0.1,smooth,thick] ({(cos((\a+pi) r)/5)/sqrt(cos(\a r)^2/25 + sin(\a r)^2)},{sin((\a+pi) r)/sqrt(cos(\a r)^2/25 + sin(\a r)^2)});
\end{tikzpicture}
\caption*{$\law(\C^{-1/2}v/|\C^{-1/2}v|)$}
\end{minipage}%
\caption{\textit{Transformation of $\mathrm{Law}(v)=\Unif(\mathbb S^1)$ used in Hit-and-Run under the change-of-coordinates $x \mapsto \C^{-1/2}x$ in the bivariate case with $\C=\diag(\kappa,1)$, $\kappa>1$.  The gray beads illustrate the redistribution of probability mass, which favors directions aligned with the high mode of $\C^{-1}$.  By elementary trigonometry, the angles $\alpha$ and $\phi(\alpha)$ in the original and transformed coordinates corresponding to the black bead are related by $\tan\phi(\alpha)=\kappa^{1/2}\tan\alpha$.  Hence, for fixed $\phi(\alpha)$ we have $\alpha\asymp\kappa^{-1/2}$, see \eqref{eq:asymp}.}}
\label{fig:wlaw}
\end{figure}

\subsection{Wasserstein contraction of generalized Hit-and-Run: Lower bound} \label{sec:wasserlower}

In the natural coordinates $\C^{-1/2}x$, the transition step consists of two distinct operations: a projection that removes the current state's component in direction $\C^{-1/2}v$, and a complete randomization of the component  using a standard Gaussian random variable $Z$, see \eqref{eq:HRstep}.  Importantly,  the auxiliary random variables $v$ and $Z$ are independent of the current state $x$. 

By synchronously coupling the auxiliary random variables in two copies of generalized Hit-and-Run, the transition of their difference reduces to a projection onto a random $(d-1)$-dimensional subspace:
\begin{equation}\label{eq:synchronous_coupling}
\mathcal{C}^{-1/2} (X - \tilde{X}) \ =\ \Pi_{\C^{-1/2}v}\,\mathcal{C}^{-1/2} (x-\tilde x) 
\end{equation}
where $x, \tilde{x} \in \mathbb{R}^d$, $X \sim \pi_{\gHR}(x, \cdot)$, $\tilde{X} \sim \pi_{\gHR}(\tilde{x}, \cdot)$ using the same $v\sim\tau$ and $Z\sim\mathcal N(0,1)$.  The law of $\C^{-1/2}v$ determines the contraction properties of this coupling, as quantified in the next lemma using the $L^2$-Wasserstein distance $\W^2_{\C^{-1/2}}$ with respect to the metric induced by $|x|_{\C^{-1/2}}=|\C^{-1/2}x|$. %corresponding to the natural coordinates.

\medskip

\begin{lemma}\label{lem:HRcontr}
It holds
\begin{equation}\label{eq:Wcontr}
    \W^2_{\C^{-1/2}}\bigr(\pi_{\gHR}(x,\cdot),\, \pi_{\gHR}(\tilde x,\cdot)\bigr)\ \leq\ (1-\rho)|x-\tilde x|_{\C^{-1/2}}\quad\text{for all $x,\tilde x\in\mathbb R^d$}
\end{equation}
with
\begin{align} \rho\ =\ \frac12\inf_{|\zeta|=1}\mathbb E_{v\sim\tau}\Bigr(\zeta\cdot\frac{\C^{-1/2}v}{|\C^{-1/2}v|}\Bigr)^2\;. \label{eq:rate}
\end{align}
\end{lemma}

The lemma provides a lower bound for the global Wasserstein contraction rate of generalized Hit-and-Run, defined as the supremum over all $\rho$ satisfying \eqref{eq:Wcontr}.  This lower bound is sharp in the sense that any $\rho'>0$ such that \eqref{eq:Wcontr} holds with $\rho'$ instead of $\rho$, satisfies $\rho'\leq3\rho$, see Lemma \ref{lem:lowerbound} below.

A lower bound for the global Wasserstein contraction rate implies a lower bound for the coarse Ricci curvature of $(\mathbb R^d,|\cdot|_{\C^{-1/2}})$ equipped with generalized Hit-and-Run which has far-reaching consequences \cite{OLLIVIER2009,JOULIN2010}.  

If $\tau$ is symmetric with respect to reflection through the origin, i.e., $\tau(A)=\tau(-A)$ for all measurable $A\subset\mathbb S^{d-1}$, then  $\mathrm{Law}(\C^{-1/2}v/|\C^{-1/2}v|)$ inherits this symmetry. In this case, the contraction rate can be expressed as
\begin{equation}\label{eq:var}
    \rho\ =\ \frac12\inf_{|\zeta|=1}\operatorname{Var}_{v\sim\tau}\Bigr(\zeta\cdot\frac{\C^{-1/2}v}{|\C^{-1/2}v|}\Bigr)\;.
\end{equation}

\begin{proof}[Proof of Lemma \ref{lem:HRcontr}]
The proof uses Lemma~\ref{lem:contrproj}, the general contraction lemma for random projections.  Let $z=\C^{-1/2}(x-\tilde x)$ and consider a synchronous coupling $(X,\tilde X)$ of $\pi_{\gHR}(x, \cdot)$ and $\pi_{\gHR}(\tilde x, \cdot)$. From \eqref{eq:synchronous_coupling},
\[ |X-\tilde X|_{\C^{-1/2}}\ =\ |\C^{-1/2}(\tilde X-X)|\ =\ |\Pi_{\C^{-1/2}v}\,z| \]
so that by Lemma \ref{lem:contrproj},
\begin{align*}
    & \W^2_{\C^{-1/2}}\bigr(\pi_{\gHR}(x, \cdot),\,\pi_{\gHR}(\tilde x, \cdot)\bigr)\ \leq\ \sqrt{\mathbb E_{v\sim\tau}|X-\tilde X|_{\C^{-1/2}}^2}\ =\ \sqrt{\mathbb E_{v\sim\tau}|\Pi_{\C^{-1/2}v}\,z|^2} \\
    &\quad\leq\ \sqrt{1-\inf_{|\zeta|=1}\mathbb E_{v\sim\tau}\Bigr(\zeta\cdot\frac{\C^{-1/2}v}{|\C^{-1/2}v|}\Bigr)^2}|z|\ \leq\ \Bigr(1-\frac12\inf_{|\zeta|=1}\mathbb E_{v\sim\tau}\Bigr(\zeta\cdot\frac{\C^{-1/2}v}{|\C^{-1/2}v|}\Bigr)^2\Bigr)|x-\tilde x|_{\C^{-1/2}}\;.
\end{align*} \end{proof}

\subsection{Wasserstein contraction of generalized Hit-and-Run: Upper bound}\label{sec:wasserupper}

We now show the Wasserstein contraction rate obtained in Lemma \ref{lem:HRcontr} for generalized Hit-and-Run to be sharp up to a multiplicative constant.
By \eqref{eq:gGibbskernel}, the averaging operator associated with generalized Hit-and-Run takes the form
\begin{equation}\label{eq:avgop}
\pi_{\gHR} f(x)\ =\ \int_{\mathbb S^{d-1}}\int_{\mathbb R}f(x+hv)\gamma^{\C}_{x,v}(dh)\tau(dv)\;.
\end{equation}
%The generator of generalized Hit-and-Run is given by
%\begin{equation}\label{eq:HRgenerator}
%\mathcal L_{\gHR} f(x)\ =\ (\pi_{\gHR}f-f)(x)\ =\ \int_{\mathbb S^{d-1}}\int_{\mathbb R}\bigr(f(x+hv)-f(x)\bigr)\gamma^{\C}_{x,v}(dh)\tau(dv)\;.
%\end{equation}
The spectral radius of this operator acting on the orthogonal complement of the constant functions in $L^2(\gamma^\C)$ is
\begin{equation}\label{eq:specrad}
1-\lambda\ =\ \sup_{f\perp1,\,\|f\|_{L^2(\gamma^\C)=1}}\bigr\|\pi_{\gHR}f\bigr\|_{L^2(\gamma^\C)}\;.
\end{equation}

\begin{lemma}\label{lem:lowerbound}
Let $\rho$ be as defined in Lemma \ref{lem:HRcontr}.
It holds
\begin{equation}\label{eq:radbound}
    \lambda\ \leq\ 3\rho\;.
\end{equation}
In particular, any $\rho'>0$ such that
\begin{equation}\label{eq:contr'}
    \W^2_{\C^{-1/2}}\bigr(\pi_{\gHR}(x,\cdot),\, \pi_{\gHR}(\tilde x,\cdot)\bigr)\ \leq\ (1-\rho')|x-\tilde x|_{\C^{-1/2}}\quad\text{for all $x,\tilde x\in\mathbb R^d$}
\end{equation}
satisfies
\begin{equation}\label{eq:rho'bound}
    \rho'\ \leq\ 3\rho\;.
\end{equation}
\end{lemma}

Note that $\rho\leq\lambda$, since the lower bound $\rho$ on the coarse Ricci curvature of $(\mathbb R^d,\,|\cdot|_{\C^{-1/2}})$ equipped with the family of probability measures $(\pi_{\gHR}(x,\cdot))_{x\in\mathbb R^d}$, implied by Lemma \ref{lem:HRcontr} and $\W^1\leq\W^2$,  yields $1-\lambda\leq1-\rho$, see \cite{ChenWang1994,OLLIVIER2009}.  Similarly, assertion \eqref{eq:rho'bound} follows from \eqref{eq:radbound} as any $\rho'>0$ satisfying \eqref{eq:contr'} provides a lower bound $\rho'$ on the coarse Ricci curvature so that
\[ 1-3\rho\ \leq\ 1-\lambda\ \leq\ 1-\rho' \]
proving \eqref{eq:rho'bound}.

\begin{proof}
We are left to show \eqref{eq:radbound}.
Define
\[ f_\zeta(x)\ =\ \zeta\cdot\C^{-1/2}x\quad\text{for $|\zeta|=1$.} \]
Since $f_\zeta\in L^2(\gamma^\C)$ with $f_\zeta\perp1$ and $\|f_\zeta\|_{L^2(\gamma^\C)}=1$, by \eqref{eq:specrad},
\begin{equation}\label{eq:gap}
    1-\lambda\ \geq\ \sup_{|\zeta|=1}\bigr\|\pi_{\gHR}f_\zeta\bigr\|_{L^2(\gamma^\C)}\;.
\end{equation}
Inserting the definition of $f_\zeta$ into \eqref{eq:avgop} and using linearity as well as \eqref{eq:sdmGaussian} yields
\begin{align*}
\pi_{\gHR}f_\zeta(x)
\ &=\ f_\zeta(x)+\int_{\mathbb S^{d-1}}\int_{\mathbb R}h\,\gamma_{x,v}^\C(dh)f_\zeta(v)\tau(dv) \\
&=\ f_\zeta(x)-\int_{\mathbb S^{d-1}}\Bigr(\frac{\C^{-1/2}v}{|\C^{-1/2}v|^2}\cdot\C^{-1/2}x\Bigr)f_\zeta(v)\tau(dv)\;.
\end{align*}
With Fubini's theorem and the fact that
\[ \int_{\mathbb R^d} (a\cdot\C^{-1/2}x) \,(b\cdot\C^{-1/2}x) \,\gamma^\C(dx)\ =\ a\cdot b\quad\text{for all $a,b\in\mathbb R^d$,} \]
it then follows
\begin{align*}
&\bigr\|\pi_{\gHR}f_\zeta\bigr\|_{L^2(\gamma^\C)}^2
\ =\ \int_{\mathbb R^d}\bigr(\pi_{\gHR}f_\zeta(x)\bigr)^2\gamma^\C(dx) \\
&=\ \|f_\zeta\|_{L^2(\gamma^\C)}^2-2\int_{\mathbb S^{d-1}}\int_{\mathbb R^d}f_\zeta(x)\Bigr(\frac{\C^{-1/2}v}{|\C^{-1/2}v|^2}\cdot\C^{-1/2}x\Bigr)\gamma^\C(dx)\,f_\zeta(v)\tau(dv) \\
&\qquad+\int_{\mathbb S^{d-1}}\int_{\mathbb S^{d-1}}\int_{\mathbb R^d}\Bigr(\frac{\C^{-1/2}v}{|\C^{-1/2}v|^2}\cdot\C^{-1/2}x\Bigr)\Bigr(\frac{\C^{-1/2}v'}{|\C^{-1/2}v'|^2}\cdot\C^{-1/2}x\Bigr)\gamma^\C(dx)\,
f_\zeta(v')\tau(dv')\,f_\zeta(v)\tau(dv) \\
&=\ 1-2\int_{\mathbb S^{d-1}}\Bigr(\zeta\cdot\frac{\C^{-1/2}v}{|\C^{-1/2}v|^2}\Bigr)f_\zeta(v)\tau(dv) \\
&\qquad+\int_{\mathbb S^{d-1}}\int_{\mathbb S^{d-1}}\Bigr(\frac{\C^{-1/2}v}{|\C^{-1/2}v|^2}\cdot\frac{\C^{-1/2}v'}{|\C^{-1/2}v'|^2}\Bigr)f_\zeta(v')\tau(dv')\,f_\zeta(v)\tau(dv) \\
&\geq\ 1-3\int_{\mathbb S^{d-1}}\Bigr(\zeta\cdot\frac{\C^{-1/2}v}{|\C^{-1/2}v|}\Bigr)^2\tau(dv)
\end{align*}
where in the last step we used the definition of $f_\zeta$ together with the triangle and Jensen's inequality
\begin{align*}
&\left|\int_{\mathbb S^{d-1}}\int_{\mathbb S^{d-1}}\Bigr(\frac{\C^{-1/2}v}{|\C^{-1/2}v|^2}\cdot\frac{\C^{-1/2}v'}{|\C^{-1/2}v'|^2}\Bigr)f_\zeta(v')\tau(dv')\,f_\zeta(v)\tau(dv)\right| \\
&\quad \leq\ \int_{\mathbb S^{d-1}}\int_{\mathbb S^{d-1}}\frac{|f_\zeta(v')|}{|\C^{-1/2}v'|}\tau(dv')\,\frac{|f_\zeta(v)|}{|\C^{-1/2}v|}\tau(dv)
\ =\ \left(\int_{\mathbb S^{d-1}}\Bigr|\zeta\cdot\frac{\C^{-1/2}v}{|\C^{-1/2}v|}\Bigr|\tau(dv)\right)^2 \\
&\quad \leq\ \int_{\mathbb S^{d-1}}\Bigr(\zeta\cdot\frac{\C^{-1/2}v}{|\C^{-1/2}v|}\Bigr)^2\tau(dv)\;.
\end{align*}
By the elementary inequality $(1-\mathsf x)^{1/2}\geq1-\mathsf x$ for $\mathsf x\in[0,1]$, then
\begin{align*}
\bigr\|\pi_{\gHR}f_\zeta\bigr\|_{L^2(\gamma^\C)}\ &\geq\ \left(1-\min\left(1,\,3\int_{\mathbb S^{d-1}}\Bigr(\zeta\cdot\frac{\C^{-1/2}v}{|\C^{-1/2}v|}\Bigr)^2\tau(dv)\right)\right)^{1/2} \\
&\geq\ 1-3\int_{\mathbb S^{d-1}}\Bigr(\zeta\cdot\frac{\C^{-1/2}v}{|\C^{-1/2}v|}\Bigr)^2\tau(dv)\;.
\end{align*}

Inserting the last display into \eqref{eq:gap} yields
\begin{align*}
1-\lambda
\ \geq\ 1-3\inf_{|\zeta|=1}\int_{\mathbb S^{d-1}}\Bigr(\zeta\cdot\frac{\C^{-1/2}v}{|\C^{-1/2}v|}\Bigr)^2\tau(dv)
\ =\ 1-3\rho
\end{align*}
with $\rho$ as defined in Lemma \ref{lem:HRcontr}, showing \eqref{eq:radbound}.
\end{proof}

\subsection{Discussion of Wasserstein contraction of Hit-and-Run}\label{sec:ballistic}

In this section, we discuss the sharp contraction rates \eqref{eq:rate} resulting from Lemma \ref{lem:HRcontr} for Hit-and-Run, i.e., $\tau = \Unif(\mathbb S^{d-1})$, in various Gaussian targets.  %Due to rotational and scale invariance of Hit-and-Run, 
By rotation and rescaling, we can assume $\C$ to be a diagonal matrix with one mode having unit variance; e.g.,  in the bivariate case, we can take $\C=\diag(\kappa,1)$ where $\kappa\geq1$ is the condition number of the corresponding Gaussian measure.  The first four cases are summarized in Table~\ref{tab:lowdimresults}.

We say that two expressions in $\kappa$ are of the same order and denote
\begin{equation}\label{eq:asymp}
f(\kappa)\ \asymp\ g(\kappa)\quad\text{if and only if}\quad c\, g(\kappa') \leq f(\kappa') \leq C\,g(\kappa')
\end{equation}
for constants $k,c,C>0$ and all $\kappa'\geq k$.
We further denote $x\propto y$ if $x=Cy$ for a constant $C>0$.

\paragraph{Bivariate Case}
Consider $\C=\diag(\kappa,1)$, $\kappa\geq1$.
Since \eqref{eq:var} holds for Hit-and-Run, among all unit vectors $\zeta$, the smallest variance is achieved for $\zeta=e_1$, as illustrated in Figure \ref{fig:wlaw}.  Thus, the contraction rate takes the form
\begin{equation}\label{eq:rhoe1}
    \rho\ =\ \frac12\mathbb E_{v\sim\sigma_1}\Bigr(e_1\cdot\frac{\C^{-1/2}v}{|\C^{-1/2}v|}\Bigr)^2\;.
\end{equation}
In particular, $e_1$ is the direction of least contraction on average under the random projection $\Pi_{\C^{-1/2}v}$.

Let $v(\alpha)=(\cos\alpha,\sin\alpha)$ and denote by $\phi(\alpha)$ the angle between $e_1$ and the transformed vector $\C^{-1/2}v(\alpha)/|\C^{-1/2}v(\alpha)|$, as illustrated in Figure \ref{fig:wlaw}.  Then, due to the symmetry of the integral,
\[ \rho\ =\ \frac12\frac2\pi\int_0^{\pi/2}\Bigr(e_1\cdot\frac{\C^{-1/2}v(\alpha)}{|\C^{-1/2}v(\alpha)|}\Bigr)^2d\alpha\ =\ \frac1\pi\int_0^{\pi/2}\cos^2\phi(\alpha)d\alpha\;. \]
Inserting $\phi(\alpha)=\arctan(\kappa^{1/2}\tan\alpha)$, the contraction rate evaluates to
\[ \rho\ =\ \frac1\pi\int_0^{\pi/2}\frac{d\alpha}{\kappa\tan^2\alpha+1}\ =\ \frac12(\kappa^{1/2}+1)^{-1}\ \asymp\ \kappa^{-1/2}\;. \]
Thus, the contraction rate is of order $\kappa^{-1/2}$.  This demonstrates that Hit-and-Run achieves ballistic contraction for bivariate Gaussian targets.

The contraction rate of Hit-and-Run can be understood geometrically as follows.  For any fix $\phi(\alpha_{\mathrm{global}})\in(0,\pi/2)$, the directions in the general $d$-dimensional case can be partitioned into two sets:
\[ \mathcal V_{\mathrm{global}}\ =\ \bigr\{v\in\mathbb S^{d-1}:|\zeta\cdot v|\geq\cos\alpha_{\mathrm{global}}\text{ for some $\zeta\in\mathcal Z$}\bigr\}\quad\text{and}\quad\mathcal V_{\mathrm{local}}\ =\ \mathbb S^{d-1}\setminus\mathcal V_{\mathrm{global}} \]
where $\mathcal Z\subset\mathbb S^{d-1}$ is the collection of directions achieving the infimum in \eqref{eq:rate}.

In the bivariate case, $\mathcal Z=\{\pm e_1\}$ so that $\mathcal V_{\mathrm{global}}$ consists of two polar-opposite spherical caps containing all directions within a $\alpha_{\mathrm{global}}$ angle around $\pm e_1$.  As $\phi(\alpha_{\mathrm{global}})$ is fix, $\alpha_{\mathrm{global}}\asymp\kappa^{-1/2}$, see Figure \ref{fig:wlaw}.  These caps, as depicted in Figure \ref{fig:Gaussiantrans}, correspond to directions producing \emph{global moves}, characterized by their ability to advance proportionally to the scale of the target distribution, or equivalently of constant order in the natural coordinates, in both modes.

The probability of choosing a direction in $\mathcal V_{\mathrm{global}}$ is given by \begin{equation}\label{eq:globprob1}
    \Unif(\mathbb S^1)(\mathcal V_{\mathrm{global}})\ \propto\ \alpha_{\mathrm{global}}\ \asymp\ \kappa^{-1/2}\;.
\end{equation}
For directions $v\in\mathcal V_{\mathrm{global}}$, their contribution to the contraction rate \eqref{eq:rhoe1}
\[ \Bigr(e_1\cdot\frac{\C^{-1/2}v}{|\C^{-1/2}v|}\Bigr)^2\ \geq\ \cos^2\phi(\alpha_{\mathrm{global}}) \]
is bounded below resulting in the ballistic rate.  It is precisely these global moves that enable Hit-and-Run to contract ballistically.  On the other hand, directions in $\mathcal V_{\mathrm{local}}$ contribute only a diffusive rate.  In particular, for $\pi/2-\alpha_{\mathrm{global}}\leq1$, the contribution from $\mathcal V_{\mathrm{local}}$ is proportional to
\begin{equation}\label{eq:localcontri}
\int_{\alpha_{\mathrm{global}}}^{\pi/2}\frac{d\alpha}{\kappa\tan^2\alpha+1}\ \leq\ \int_{\alpha_{\mathrm{global}}}^{\pi/2}\frac{d\alpha}{\frac{\kappa}{4(\pi/2-\alpha)^2}+1}\ \leq\ \frac{\frac\pi2-\alpha_{\mathrm{global}}}{\frac{\kappa}{4(\pi/2-\alpha_{\mathrm{global}})^2}+1}\ \asymp\ \kappa^{-1}\Bigr(\frac\pi2-\alpha_{\mathrm{global}}\Bigr)^3\ \leq\ \kappa^{-1}\;.
\end{equation}
Thus, restricting to directions in $\mathcal V_{\mathrm{local}}$ yields only  a diffusive rate.

\paragraph{Three dimensions: One low mode}

Consider $\C=\diag(\kappa,1,1)$, i.e., a three-dimensional Gaussian with two high and one low mode with respect to $\C^{-1}$.  Due to rotational symmetry around $e_1$, Figure \ref{fig:wlaw} carries over to the distribution of directions on $\mathbb S^2$ so that $\mathcal Z$ remains $\{\pm e_1\}$ and
\[ \rho\ =\ \frac12\mathbb E_{v\sim\sigma_2}\Bigr(e_1\cdot\frac{\C^{-1/2}v}{|\C^{-1/2}v|}\Bigr)^2\;. \]
However, since the surface area of a spherical cap on $\mathbb S^2$ scales quadratically in its radius, in contrast to the linear scaling encountered in \eqref{eq:globprob1}, it holds
\[ \Unif(\mathbb S^2)(\mathcal V_{\mathrm{global}})\ \propto\ \alpha_{\mathrm{global}}^2\ \asymp\ \kappa^{-1} \]
suggesting global moves to appear too seldom for ballistic contraction.  Indeed, the rate computes to
\[ \rho\ =\ \frac12\frac2{4\pi}\int_0^{\pi/2}\cos^2\phi(\alpha)2\pi\sin\alpha d\alpha \ =\ \frac12\int_0^{\pi/2}\frac{\sin\alpha}{\kappa\tan^2\alpha+1}d\alpha\ \asymp\ \kappa^{-1}\log\kappa \]
uncovering a superdiffusive scaling.

\paragraph{Three dimensions: One high mode}

Replacing one high mode by another low mode in the previous case, i.e., considering $\C=\diag(\kappa,\kappa,1)$, yields a three-dimensional Gaussian with rotational symmetry around $e_3$.  Then, the set of directions achieving the infimum in \eqref{eq:rate} becomes $\mathcal Z=\{\zeta\in\mathbb S^2:\zeta\cdot e_3=0\}$, which corresponds to the equator of the sphere.  This expands $\mathcal V_{\mathrm{global}}$ to a neighborhood of the equator, for which, similarly to \eqref{eq:globprob1},
\[ \Unif(\mathbb S^2)(\mathcal V_{\mathrm{global}})\ \propto\ \alpha_{\mathrm{global}}\ \asymp\ \kappa^{-1/2}\;. \]
As a result, Hit-and-Run again achieves ballistic contraction due to global moves with rate
\[ \rho\ \asymp\ \kappa^{-1/2}\;. \]

\paragraph{Four dimensions: One low mode}

We found that in three dimensions with only one low mode, the contraction rate becomes superdiffusive.  To understand how this behavior changes with increasing dimension, we consider the four-dimensional Gaussian with one low mode, i.e., $\C=\diag(\kappa,1,1,1)$.  In this case, the infimum in the rate is again realized in $\zeta=e_1$. Using spherical coordinates, 
\begin{align*}
    \rho\ =\ \frac12\mathbb E_{v\sim\sigma_3}\frac{\kappa^{-1}v_1^2}{|v|^2-(1-\kappa^{-1})v_1^2}\ =\ \frac1\pi\int_0^\pi\frac{\kappa^{-1}\cos^2s\sin^2s}{1-(1-\kappa^{-1})\cos^2s}ds\ =\ \frac{\kappa^{-1}}{2(1+\kappa^{-1/2})^2}\ \asymp\ \kappa^{-1}\;.
\end{align*}
Thus, the contraction rate is diffusive.
This suggests that Hit-and-Run contracts diffusively in high-dimensional Gaussians when the number of low modes is small.

\paragraph{High dimension}
After understanding that Hit-and-Run contracts diffusively in low dimensions with too few low modes, we examine high-dimensional, two-scale Gaussians with $d_1$ low and $d_2$ high modes, i.e.,
\[ \C\ =\ \begin{pmatrix} \kappa I_{d_1} & 0 \\ 0 & I_{d_2}\end{pmatrix}\;. \]
The contraction rate in \eqref{eq:rate} can be expressed as
\[ \rho\ =\ \frac12\inf_{|\zeta|=1}\mathbb E_{v\sim\gamma_{d_1+d_2}}\Bigr(\zeta\cdot\frac{\C^{-1/2}v}{|\C^{-1/2}v|}\Bigr)^2\ =\ \frac12\inf_{|\zeta^1|^2+|\zeta^2|^2=1}\mathbb E_{v^1\sim\gamma_{d_1}}\mathbb E_{v^2\sim\gamma_{d_2}}\frac{(\kappa^{-1/2}\zeta^1\cdot v^1+\zeta^2\cdot v^2)^2}{\kappa^{-1}|v^1|^2+|v^2|^2} \]
where we replaced the expectation over $\sigma_{d_1+d_2-1}$ with an expectation over the $(d_1+d_2)$-dimensional canonical Gaussian measure $\gamma_{d_1+d_2}$ normalized to the unit sphere.
When $d_1$ and $d_2$ are large,  $|v^1|^2\approx d_1$ and $|v^2|^2\approx d_2$ with high probability for $v^1\sim\gamma_{d_1}$ and $v^2\sim\gamma_{d_2}$.  Substituting these approximations into the rate, we have
\begin{align*}
    \rho\ &\approx\ \frac12\inf_{|\zeta^1|^2+|\zeta^2|^2=1}\frac{\kappa^{-1}\mathbb E_{v^1\sim\gamma_{d_1}}(\zeta^1\cdot v^1)^2+\mathbb E_{v^2\sim\gamma_{d_2}}(\zeta^2\cdot v^2)^2}{\kappa^{-1}d_1+d_2}
    =\ \frac12\inf_{|\zeta^1|^2+|\zeta^2|^2=1}\frac{\kappa^{-1}|\zeta^1|^2+|\zeta^2|^2}{\kappa^{-1}d_1+d_2}\\
    &=\ \frac12\frac{\kappa^{-1}}{\kappa^{-1}d_1+d_2}
    \ =\ \frac12(d_1+\kappa d_2)^{-1}\;.
\end{align*}
We therefore expect Hit-and-Run to contract ballistically only in the regime $d_2=\mathcal O(\kappa^{-1/2}d_1)$, i.e., if the relative number of high modes $d_2$ is sufficiently small relative to the number of low modes $d_1$.

\subsection{Mixing of Hit-and-Run}\label{sec:mixing}

After proving and discussing the sharp Wasserstein contraction with rate $\rho$ of Hit-and-Run, we now combine this result with a one-step overlap bound to obtain a mixing time upper bound.  The following theorem states that Hit-and-Run mixes at a rate proportional to its Wasserstein contraction rate, i.e., the mixing time is of order $\rho^{-1}$ up to logarithmic factors.  Let $\mathcal P(\mathbb R^d)$ denote the collection of all probability measures on $\mathbb R^d$.

\medskip

\begin{theorem}
\label{thm:HRmixing}
Let $\eps>0$, $\nu\in\mathcal P(\mathbb R^d)$ and $\rho$ as defined in \eqref{eq:rate}.
Then, there exist absolute constants $C,C'>0$ such that the mixing time of Hit-and-Run with respect to the Gaussian measure $\gamma^\C$ starting from $\nu$ to accuracy $\eps$ satisfies
\begin{equation}\label{eq:thm_mix}
\begin{aligned}
    \tmix(\eps,\nu)\ &=\ \inf\bigr\{ n \in \mathbb{N}\,:\,\mathcal \TV\bigr(\nu\pi_{\HR}^n,\,\gamma^\C\bigr)\leq \eps \bigr\} \\
    &\leq\  C\,\rho^{-1}\,\log\Bigr(C'\max(\kappa,M,m^{-1})\,d\,\eps^{-1}\,\W^2_{\C^{-1/2}}(\gamma^{\C}, \nu)\Bigr)
\end{aligned}
\end{equation}
where $\kappa,M,m$ denote condition number, largest and smallest eigenvalue of $\C^{-1}$.
\end{theorem}

In Section~\ref{sec:overlap}, we prove a one-step overlap bound that essentially relates the total variation distance after a transition of Hit-and-Run to the Wasserstein distance prior to the transition.  Using this result, the total variation distance to stationarity after $n$ transitions can be controlled in terms of the Wasserstein distance to stationarity after $n-1$ transitions, which decays exponentially with rate proportional to $\rho$.  This leads to the mixing time upper bound stated in Theorem \ref{thm:HRmixing}, whose detailed proof is given at the end of Section \ref{sec:overlap}.  Similar approaches to bounding mixing times have been employed in previous works  \cite{roberts2002one,madras2010quantitative,EbMa2019,monmarche2021high, BouRabeeEberle2023,BouRabeeOberdoerster2024}.

%The proof of this result relies on the coupling lemma
%\begin{equation}\label{eq:coupling_lemma}
%\TV\bigr(\nu\pi_{\HR}^n,\,\gamma^\C\bigr) \  = \ \TV\bigr(\nu\pi_{\HR}^n,\,\gamma^\C\pi_{\HR}^n\bigr) \  \leq \  \mathbb P(X_n\ne\widetilde X_n)
%\end{equation} 
%where $X_n$ and $\widetilde X_n$ are two copies of  Hit-and-Run starting from $\gamma^\C$ and $\nu$, respectively.  To estimate the probability of coalescence between the two copies, we use a contractive coupling for the first $n-1$ steps, followed by a one-shot coupling at step $n$.

\subsection{One-Step Overlap of Hit-and-Run}

\label{sec:overlap}

Here, we develop a total variation upper bound, which shows the transitions of Hit-and-Run in the target distribution $\gamma^\C$ to have a regularizing effect.  Even for the Gaussian target measure, the transitions of Hit-and-Run are not Gaussian, see Figure~\ref{fig:hr_densities}.  As a result, existing formulas cannot be directly applied, and new arguments must be developed.

\medskip

\begin{lemma} \label{lem:overlap1}
	For any $\epsilon \in (0,1)$ and  $x,\tilde x\in\mathbb R^d$,
    \[ \TV\bigr(\pi_{\HR}(x,\cdot),\,\pi_{\HR}(\tilde{x},\cdot)\bigr)
    \ \leq\ \sqrt 2\Bigr(C_1(x)^{1/2}\bigr|x-\tilde{x}\bigr|_{\C^{-1/2}} + C_2(x)^{1/2}\bigr|x-\tilde{x}\bigr|_{\C^{-1/2}}^{1/2} + C_3(x)\epsilon^{1/2}\Bigr) \]
    where we have introduced
    \begin{align*}
    C_1(x) \ &= \ \epsilon^{-1}m^{-1/2}|x|_{\C^{-1/2}} + 2\epsilon M^{1/2} + 1   \;, \\ 
    C_2(x) \ &= \ 2\epsilon^{-1}m^{-1/2}|x|_{\C^{-1/2}}^2 + 2|x|_{\C^{-1/2}} + \bigr(\epsilon^{-1}m^{-1/2} + 2\epsilon\kappa^{1/2}\bigr)(d-1) \\
    &\qquad+ 2\epsilon^{-1}m^{-1/2}+\bigr( M^{1/2} + m^{-1/2} +2 \bigr) \kappa^{1/2}  \;, \\ 
    C_3(x) \ &= \ \sqrt3M^{1/4}|x|_{\C^{-1/2}} + \sqrt2(1+\log\epsilon^{-1})^{1/2}M^{1/4}\sqrt{d-1} + \sqrt2M^{1/4}\bigr( M^{1/2} + m^{-1/2} +2 \bigr)^{1/2}   \;.
    \end{align*}
\end{lemma}

The proof uses Pinsker's inequality, which states that for any probability measures $\nu ,\eta \in \mathcal{P}(\mathbb{R}^d)$,
\begin{equation}\label{eq:Pinsker}
    \TV(\eta,\,\nu)\ \leq\  \sqrt{2}\,\mathcal H(\eta|\nu)^{1/2} \;,
\end{equation}
where the relative entropy of $\eta$ with respect to $\nu$, $H(\eta\:|\:\nu)\in[0,\infty]$, is defined as
\begin{equation}\label{eq:relentropie}
\mathcal H(\eta\:|\:\nu) \ = \ 
\begin{cases}
\int\log\frac{d\eta}{d\nu}\: d\eta  & \text{if $\eta\ll\nu$,}\\
\infty & \text{otherwise.}
\end{cases}
\end{equation}

To compute the relative entropy between two Hit-and-Run transition kernels starting at different initial points, we use the transition kernel's representation \eqref{eq:HRcart} in Cartesian coordinates and insert \eqref{eq:sdmGaussian}, which yields
\begin{align}
& \pi_{\HR}(x,dy) \ = \ 1_{ \{ y \ne x \} } \, \frac{2}{ a_{d-1}}  \, \frac{1}{\left|y-x \right|^{d-1}} \, \gamma^{\C}_{x,\frac{y-x}{|y-x|}}(|y-x|)  \,  dy \nonumber \\
& \ = \ 1_{ \{ y \ne x \} } \, \frac{2}{ a_{d-1}} \, \frac{\left|\C^{-1/2} (y-x) \right|}{\sqrt{2 \pi} \left|y-x \right|^{d} } \, \exp\left( -\frac{\bigr|\C^{-1/2} y \bigr|^2}{2} + \frac{\bigr|\C^{-1/2} x \bigr|^2}{2} - \frac{(x \cdot \C^{-1} (y-x))^2}{2 \bigr|\C^{-1/2} (y-x) \bigr|^2} \right)\, dy\;.   \label{eq:xyHRkernel}
\end{align}

%\begin{comment}
\begin{figure}
\centering
\includegraphics[width=\textwidth]{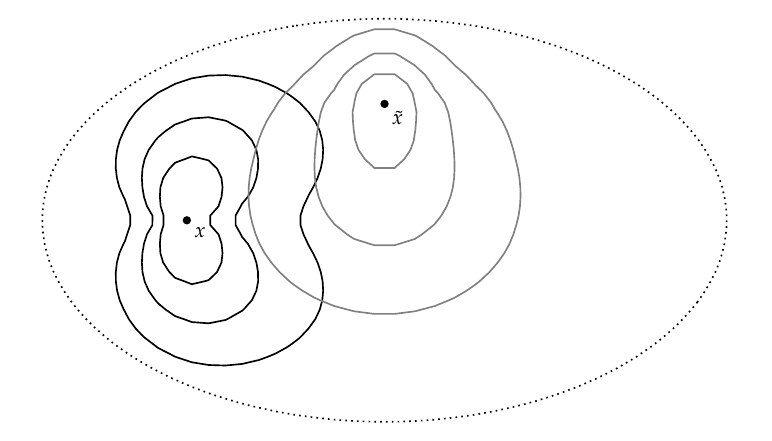}
\caption{\textit{This figure shows contour lines of the transition density of the Hit-and-Run kernel in \eqref{eq:xyHRkernel} starting from two different points: $x=(-2,0)$ (black) and $\tilde{x}=(0,1)$ (gray).  The covariance matrix is $\mathcal{C}=\diag(4,1)$. A contour line of the target density (dotted) is included for comparison.}} \label{fig:hr_densities}
\end{figure}
%\end{comment}

\begin{proof}
 Using \eqref{eq:xyHRkernel}, we  decompose the relative entropy of $\pi_{\HR}(x,\cdot)$ with respect to $\pi_{\HR}(\tilde{x},\cdot)$ as\begin{align}
 \mathcal H \bigr( & \pi_{\HR}(x,\cdot) |\pi_{\HR}(\tilde{x},\cdot)\bigr) \ = \ \rn{1} + \rn{2} + \rn{3} + \rn{4} \quad \text{where} \label{eq:re} \\
 \rn{1}  \ &= \ \frac{|\C^{-1/2} x |^2}{2} - \frac{|\C^{-1/2} \tilde{x} |^2}{2}  \;, \nonumber \\
 \rn{2}  \ &= \  \int_{\mathbb{R}^{d}} \left(\frac{(\tilde{x} \cdot \C^{-1} (y-\tilde{x}))^2}{2 \left|\C^{-1/2} (y-\tilde{x}) \right|^2} - \frac{(x \cdot \C^{-1} (y-x))^2}{2 \left|\C^{-1/2} (y-x) \right|^2} \right)  \pi_{\HR}(x,dy) \;, \nonumber \\
  \rn{3} \ &= \  (d -1) \int_{\mathbb{R}^{d}} \log\frac{|y-\tilde{x}|}{\left|y-x \right|} \, \pi_{\HR}(x,dy) \;, \nonumber \\
 \rn{4} \ &= \  \int_{\mathbb{R}^{d}}  \left( \log\frac{|\C^{-1/2}(y-x)|}{|y-x|} + \log\frac{\left|y-\tilde{x} \right|}{\left|\C^{-1/2} (y-\tilde{x}) \right|}   \right) \,  \pi_{\HR}(x,dy)  \;. \nonumber 
\end{align} 
The first term, $\rn{1}$, can be estimated directly using the  reverse triangle inequality:
\begin{align}
\rn{1}  \ &= \ \frac{1}{2} \bigr( |\C^{-1/2} x | + |\C^{-1/2} \tilde{x} | \bigr) \bigr(|\C^{-1/2} x | -  |\C^{-1/2} \tilde{x} | \bigr) \;,  \nonumber \\
&\leq \ \frac{1}{2} \Bigr( 2 | \C^{-1/2} x | + \bigr|\C^{-1/2} (x-\tilde{x}) \bigr| \Bigr)  \bigr|\C^{-1/2} (x-\tilde{x}) \bigr| \;. \label{eq:re1} 
\end{align} 
To bound the remaining terms in \eqref{eq:re}, we split each integral into two parts using an indicator function for the set $A_{x,\epsilon} = \{ y \in \mathbb{R}^d : |y-x| > \epsilon \} $ and its complement $A_{x,\epsilon}^c = \mathbb{R}^d \setminus A_{x,\epsilon}$, where $\epsilon>0$.
Additionally, from \eqref{eq:H&Rkernel} and \eqref{eq:sdmGaussian}, we can bound the probability to transition to $A_{x,\epsilon}^c$ as follows:
\begin{align}
 & \pi_{\HR}(x, A_{x,\epsilon}^c)
 \ =\ \int_{\mathbb{S}^{d-1}}\int_{\mathbb{R}} 1_{ A_{x,\epsilon}^c } \left(x + h v \right) \, \gamma^\C_{x,v}(dh) \, \sigma_{d-1}(dv)
 \ =\ \int_{\mathbb{S}^{d-1}}\int_{-\epsilon}^\epsilon \gamma^\C_{x,v}(dh) \, \sigma_{d-1}(dv)\nonumber\\
 &\qquad\leq\ \int_{\mathbb{S}^{d-1}}\int_{-\epsilon}^{\epsilon}\,dh \, |\C^{-1/2}v|\, \sigma_{d-1}(dv)
 \ \leq\ 2\epsilon\sup\nolimits_{|v|=1}|\C^{-1/2}v|
 \ \leq\ 2\epsilon M^{1/2}\;,  \label{eq:Axepsic2b}
\end{align}  
where $M^{1/2}$ is the largest eigenvalue of $\C^{-1/2}$.  Similarly, we can bound
\begin{align}
& \int_{A_{x,\epsilon}^c}\log|y-x|^{-1}\pi_{\HR}(x,dy)
\ =\ \int_{\mathbb{S}^{d-1}}\int_{-\epsilon}^\epsilon\log|h|^{-1} \, \gamma^\C_{x,v}(dh) \, \sigma_{d-1}(dv)\nonumber \\
&\qquad\leq\ \int_{\mathbb{S}^{d-1}}\int_{-\epsilon}^\epsilon\log|h|^{-1} \, dh \, |\C^{-1/2}v| \, \sigma_{d-1}(dv)
\ \leq\ 2\epsilon(1-\log\epsilon)M^{1/2}\;.
\label{eq:Axepsic4b}
\end{align}
The bounds \eqref{eq:Axepsic2b} and \eqref{eq:Axepsic4b} will be useful to estimate the remaining terms in the region $A_{x,\epsilon}^c$.

Specifically, we decompose the second term, \rn{2}, as follows:
\begin{align}
\rn{2}  \ &= \   \rn{2}_a + \rn{2}_b \quad \text{where} \label{eq:re2_split} \\
\rn{2}_a \ &= \ \frac12\int_{A_{x,\epsilon}} \left( \frac{(\tilde{x} \cdot \C^{-1} (y-\tilde{x}))^2}{\left|\C^{-1/2} (y-\tilde{x}) \right|^2} - \frac{(x \cdot \C^{-1} (y-x))^2}{\left|\C^{-1/2} (y-x) \right|^2} \right)  \pi_{\HR}(x,dy) \;, \nonumber \\    
\rn{2}_b \ &= \ \frac12\int_{A_{x,\epsilon}^c} \left( \frac{(\tilde{x} \cdot \C^{-1} (y-\tilde{x}))^2}{\left|\C^{-1/2} (y-\tilde{x}) \right|^2} - \frac{(x \cdot \C^{-1} (y-x))^2}{\left|\C^{-1/2} (y-x) \right|^2} \right)  \pi_{\HR}(x,dy) \;. \nonumber 
\end{align}
The term $\rn{2}_a$ involves a difference of two squares. The first factor in this difference can be bounded by using the Cauchy-Schwarz and triangle inequality: \begin{align}
\rn{2}_a \ &= \    \frac{1}{2} \int_{A_{x,\epsilon}} \left(\frac{\tilde{x} \cdot \C^{-1} (y-\tilde{x})}{ \left|\C^{-1/2} (y-\tilde{x}) \right|}  + \frac{x \cdot \C^{-1} (y-x)}{ \left|\C^{-1/2} (y-x) \right|}\right) \nonumber \\
& \qquad\qquad\qquad\qquad \times \left( \frac{\tilde{x} \cdot \C^{-1} (y-\tilde{x})}{ \left|\C^{-1/2} (y-\tilde{x}) \right|} -  \frac{x \cdot \C^{-1} (y-x)}{ \left|\C^{-1/2} (y-x) \right|} \right)   \pi_{\HR}(x,dy) \;, \nonumber \\ 
 \ &\le \ \frac{1}{2} \Bigr( 2|\C^{-1/2}x| + \bigr|\C^{-1/2} (x-\tilde{x}) \bigr| \Bigr) \nonumber \\ 
& \qquad\qquad\qquad\qquad \times  \int_{A_{x,\epsilon}} \left| \frac{\tilde{x} \cdot \C^{-1} (y-\tilde{x})}{ \left|\C^{-1/2} (y-\tilde{x}) \right|} -  \frac{x \cdot \C^{-1} (y-x)}{ \left|\C^{-1/2} (y-x) \right|} \right| \pi_{\HR}(x,dy) \;. \label{eq:re2a0}
\end{align}
The integrand in \eqref{eq:re2a0} can also be bounded by using the Cauchy-Schwarz and triangle inequalities
\begin{align}
& \left| \frac{\tilde{x} \cdot \C^{-1} (y-\tilde{x})}{ \left|\C^{-1/2} (y-\tilde{x}) \right|} -  \frac{x \cdot \C^{-1} (y-x)}{ \left|\C^{-1/2} (y-x) \right|} \right|
\ \leq\ \bigr|\C^{-1/2} (x-\tilde{x}) \bigr| +  \left| \frac{x \cdot \C^{-1} (y-\tilde{x})}{ \left|\C^{-1/2} (y-\tilde{x}) \right|} -  \frac{x \cdot \C^{-1} (y-x)}{ \left|\C^{-1/2} (y-x) \right|} \right|  \nonumber \\
&\qquad \leq\ \bigr|\C^{-1/2} (x-\tilde{x}) \bigr| + |\C^{-1/2} x| \left| \frac{\C^{-1/2} (y-\tilde{x})}{ \left|\C^{-1/2} (y-\tilde{x}) \right|} -  \frac{\C^{-1/2} (y-\tilde{x})}{ \left|\C^{-1/2} (y-x) \right|} + \frac{\C^{-1/2} (x-\tilde{x})}{ \left|\C^{-1/2} (y-x) \right|} \right|  \nonumber \\
&\qquad \leq\ \bigr|\C^{-1/2} (x-\tilde{x}) \bigr| + |\C^{-1/2} x| \left( \left| \frac{\C^{-1/2} (y-\tilde{x})}{ \left|\C^{-1/2} (y-\tilde{x}) \right|} -  \frac{\C^{-1/2} (y-\tilde{x})}{ \left|\C^{-1/2} (y-x) \right|} \right| + \left| \frac{\C^{-1/2} (x-\tilde{x})}{ \left|\C^{-1/2} (y-x) \right|} \right| \right) \nonumber \\
&\qquad =\ \bigr|\C^{-1/2} (x-\tilde{x}) \bigr| + \frac{|\C^{-1/2} x|}{\left|\C^{-1/2} (y-x) \right|} \left( \Bigr| \bigr|\C^{-1/2} (y-x) \bigr|  -  \bigr|\C^{-1/2} (y-\tilde{x}) \bigr| \Bigr| + \bigr| \C^{-1/2} (x-\tilde{x}) \bigr| \right) \nonumber \\
&\qquad  \leq\ \Bigr(1 + \frac{2}{\epsilon m^{1/2}} |\C^{-1/2} x|\Bigr) \bigr|\C^{-1/2} (x-\tilde{x}) \bigr| \label{eq:re2aintegrand}
\end{align}
where in the last step we used that $|\C^{-1/2} (y-x)| \geq \epsilon m^{1/2}$ in $A_{x,\epsilon}$.
Inserting the bound \eqref{eq:re2aintegrand} back into \eqref{eq:re2a0} gives
\begin{align}
\rn{2}_a \ &\leq\ \frac{1}{2} \Bigr( 2| \C^{-1/2} x| + \bigr|\C^{-1/2} (x-\tilde{x}) \bigr| \Bigr)    \Bigr(1 + \frac{2}{\epsilon m^{1/2}} | \C^{-1/2} x | \Bigr) \bigr|\C^{-1/2} (x-\tilde{x}) \bigr| \nonumber \\
\ &\le \  \Bigr( | \C^{-1/2} x | + \frac{2}{\epsilon m^{1/2}} | \C^{-1/2} x |^2 \Bigr) \bigr|\C^{-1/2} (x-\tilde{x}) \bigr| + \frac{1}{2} \Bigr(1 + \frac{2}{\epsilon m^{1/2}} | \C^{-1/2} x | \Bigr)  \bigr|\C^{-1/2} (x-\tilde{x}) \bigr|^2 \;. \label{eq:re2a}
\end{align}

On the other hand, by again using the Cauchy-Schwarz, triangle inequalities and \eqref{eq:Axepsic2b}
\begin{align}
\rn{2}_b \ &\leq \ \frac12\int_{A_{x,\epsilon}^c} \left(  \left|\frac{\tilde{x} \cdot \C^{-1} (y-\tilde{x}))}{\left|\C^{-1/2} (y-\tilde{x}) \right|}\right|^2 + \left|\frac{x \cdot \C^{-1} (y-x))}{\left|\C^{-1/2} (y-\tilde{x}) \right|}\right|^2 \right)  \pi_{\HR}(x,dy) \nonumber\\
&\leq \ \frac{1}{2}\bigr(|\C^{-1/2} \tilde{x}|^2+|\C^{-1/2}x|^2\bigr) \, \pi_{\HR}(x,A_{x,\epsilon}^c)\ \leq\ 2\epsilon M^{1/2}\bigr|\C^{-1/2}(x-\tilde x)\bigr|^2+3\epsilon M^{1/2}|\C^{-1/2}x|^2\;.
\label{eq:re2b}
\end{align}
Inserting \eqref{eq:re2a} and \eqref{eq:re2b} into \eqref{eq:re2_split} yields \begin{equation} \label{eq:re2}
\begin{aligned}
& \rn{2} \ \leq \ \Bigr( | \C^{-1/2} x | + \frac{2}{\epsilon m^{1/2}} | \C^{-1/2} x |^2 \Bigr) \bigr|\C^{-1/2} (x-\tilde{x}) \bigr| \\
& \qquad  + \Bigr(\frac{1}{2}  + 2 \epsilon M^{1/2}  + \frac{1}{\epsilon m^{1/2}} | \C^{-1/2} x | \Bigr)  \bigr|\C^{-1/2} (x-\tilde{x}) \bigr|^2 +  3 \epsilon M^{1/2}  | \C^{-1/2} x |^2 \;.
\end{aligned}
\end{equation}

Likewise, we can decompose the third term, \rn{3}, as: \begin{align}
 \rn{3}  \ &= \   \rn{3}_a + \rn{3}_b \quad \text{where} \label{eq:re3_split} \\
 \rn{3}_a \ &= \ (d - 1) \int_{A_{x,\epsilon}} \log\frac{\left|y-\tilde{x} \right|}{\left|y-x \right|}  \pi_{\HR}(x,dy) \;, \nonumber \\    
 \rn{3}_b \ &= \ (d-1) \int_{A_{x,\epsilon}^c} \log\frac{\left|y-\tilde{x} \right|}{\left|y-x \right|}  \pi_{\HR}(x,dy) \;. \nonumber 
\end{align}
Applying the elementary inequality \begin{equation}
\label{ieq:log}
\log(\sf{x})\ \leq\ \sf{x} - 1  \quad \text{for $\sf{x} > 0$} 
\end{equation}
in $\rn{3}_a$ gives
\begin{align}
\rn{3}_a \ &\leq \ (d-1) \int_{A_{x,\epsilon}} \left( \frac{\left|y-\tilde{x} \right|}{\left|y-x \right|} - 1 \right)  \pi_{\HR}(x,dy)
\ \leq  \   (d -1)  \int_{A_{x,\epsilon}} \frac{\bigr||y-\tilde{x}|-|y-x|\bigr|}{|y-x|}   \pi_{\HR}(x,dy) \nonumber \\ 
\ &\le \  (d-1) \frac{|x-\tilde{x}|}{\epsilon}
\ \leq\ \frac{(d-1)m^{-1/2}}{\epsilon} \bigr|\C^{-1/2} (x-\tilde{x}) \bigr|  \;, \label{eq:re3a}
\end{align}
where we used the triangle inequality, $|y-x|\geq\epsilon$ for $y \in A_{x,\epsilon}$, and $|x-\tilde{x}| \leq m^{-1/2} |\C^{-1/2} (x-\tilde{x})|$.

On the other hand, for $y \in A_{x,\epsilon}^c$, we have $|y-x|\leq\epsilon$.  In this case, we further decompose $\rn{3}_b=\rn{3}_{b,1}+\rn{3}_{b,2}$ where
\begin{align*}
\rn{3}_{b,1}\ &=\ (d-1) \int_{A_{x,\epsilon}^c} \log|y-\tilde{x}|\,\pi_{\HR}(x,dy) \;, \\ 
\rn{3}_{b,2}\ &=\ (d-1) \int_{A_{x,\epsilon}^c} \log|y-x|^{-1}\pi_{\HR}(x,dy) \;. 
\end{align*} 
For the first term, we use $|y-x| \leq \epsilon \leq 1$, as follows:
\begin{align}
\rn{3}_{b,1} \ &\leq\ (d-1)\int_{A_{x,\epsilon}^c}\log|y-\tilde{x}|\,\pi_{\HR}(x,dy)
\ \leq\ (d-1)\int_{A_{x,\epsilon}^c}\log\bigr(|y-x|+|x-\tilde{x}|\bigr)\pi_{\HR}(x,dy) \nonumber \\
&\leq\  (d-1)\log\bigr(1 + |x-\tilde{x}| \bigr)\pi_{\HR}(x,A_{x,\epsilon}^c)
\ \leq\ 2\epsilon M^{1/2}(d-1)|x-\tilde{x}| \nonumber\\
&\leq\ 2\epsilon\kappa^{1/2}(d-1)\bigr|\C^{-1/2} (x-\tilde{x})\bigr|\;, \label{eq:re3b1}
\end{align}
where we used \eqref{ieq:log}, \eqref{eq:Axepsic2b}, and $|x-\tilde{x}| \leq m^{-1/2} |\C^{-1/2} (x-\tilde{x})|$.
For the second term, by \eqref{eq:Axepsic4b}
\begin{align}
\rn{3}_{b,2}\ \leq\ (d-1)\int_{A_{x,\epsilon}^c}\log|y-x|^{-1}\pi_{\HR}(x,dy)
\ \leq \ 2\epsilon(1-\log\epsilon)M^{1/2}(d-1)\;. \label{eq:re3b2}
\end{align}
Inserting \eqref{eq:re3a}, \eqref{eq:re3b1}, and \eqref{eq:re3b2} into \eqref{eq:re3_split} yields
\begin{equation} \label{eq:re3}
\rn{3} \ \leq \ (d-1)\Bigr(\bigr(\epsilon^{-1}m^{-1/2} + 2\epsilon\kappa^{1/2}\bigr)\bigr|\C^{-1/2} (x-\tilde{x})\bigr| + 2\epsilon(1-\log\epsilon)M^{1/2}\Bigr)\;.
\end{equation}

Similarly, we can decompose the fourth term, \rn{4}, as
\begin{align}
 \rn{4}  \ &= \   \rn{4}_a + \rn{4}_b \quad \text{where} \label{eq:re4_split} \\
 \rn{4}_a \ &= \ \int_{A_{x,\epsilon}}  \left( \log\frac{|\C^{-1/2}(y-x)|}{|y-x|} + \log\frac{\left|y-\tilde{x} \right|}{\left|\C^{-1/2} (y-\tilde{x}) \right|}   \right) \,  \pi_{\HR}(x,dy) \;, \nonumber \\    
 \rn{4}_b \ &= \ \int_{A_{x,\epsilon}^c}  \left( \log\frac{|\C^{-1/2}(y-x)|}{|y-x|} + \log\frac{\left|y-\tilde{x} \right|}{\left|\C^{-1/2} (y-\tilde{x}) \right|}   \right) \,  \pi_{\HR}(x,dy) \;. \nonumber 
\end{align}
For $y \in A_{x,\epsilon}$, we have $|y-x| \geq \epsilon$.
In this case, we further decompose $\rn{4}_a=\rn{4}_{a,1}+\rn{4}_{a,2}$ as follows:
\begin{align*}
\rn{4}_{a,1} \ &= \   \int_{A_{x,\epsilon}\cap A_{\tilde{x}, \epsilon}}  \left(\log\frac{|\C^{-1/2}(y-x)|}{|\C^{-1/2}(y-\tilde{x})|} + \log \frac{|y-\tilde{x}|}{|y-x|}\right)\pi_{\HR}(x,dy) \;,    \\
\rn{4}_{a,2} \ &= \  \int_{A_{x,\epsilon}\cap A_{\tilde{x}, \epsilon}^c} \left(\log\frac{|\C^{-1/2}(y-x)|}{|y-x|} + \log\frac{|y-\tilde{x}|}{|\C^{-1/2}(y-\tilde{x})|}\right)\pi_{\HR}(x,dy) \;. 
\end{align*}
By applying \eqref{ieq:log} and the triangle inequality, the first term can be bounded by
\begin{align}
\rn{4}_{a,1} \ &\leq \ \int_{A_{x,\epsilon}\cap A_{\tilde{x}, \epsilon}}\left(\frac{\left|\C^{-1/2} (y-x) \right| - \left|\C^{-1/2} (y-\tilde{x}) \right|}{\left|\C^{-1/2} (y-\tilde{x}) \right|} + \frac{\left|y-\tilde{x} \right| - \left|y-x \right|}{\left|y-x \right|}\right) \pi_{\HR}(x,dy)  \nonumber \\
&\leq \ \int_{A_{x,\epsilon}\cap A_{\tilde{x}, \epsilon}}\left(\frac{\left|\C^{-1/2} (x-\tilde{x}) \right|}{\left|\C^{-1/2} (y-\tilde{x}) \right|} + \frac{\left|x-\tilde{x} \right|}{\left|y-x \right|} \right) \pi_{\HR}(x,dy) \nonumber \\
&\leq \ \left(\frac{|\C^{-1/2} (x-\tilde{x})|}{\epsilon m^{1/2}} + \frac{|x-\tilde{x}|}{\epsilon}\right)\pi_{\HR}(x,A_{x,\epsilon}\cap A_{\tilde{x}, \epsilon})
\ \leq \  \frac{2}{\epsilon m^{1/2}}\bigr|\C^{-1/2}(x-\tilde{x})\bigr| \;.
\label{eq:re4a1}
\end{align}
For the second term, by again applying \eqref{ieq:log}, and subsequently using that $\epsilon<|y-x|\leq\epsilon+|x-\tilde x|$ for $y\in A_{x,\epsilon}\cap A_{\tilde{x}, \epsilon}^c$ in an estimate resembling \eqref{eq:Axepsic2b},
\begin{align}
\rn{4}_{a,2} \ &\leq \ \int_{A_{x,\epsilon}\cap A_{\tilde{x}, \epsilon}^c}\left(\frac{|\C^{-1/2}(y-x)|}{|y-x|}+\frac{|y-\tilde{x}|}{|\C^{-1/2} (y-\tilde{x})|}-2\right) \pi_{\HR}(x,dy)  \nonumber \\  
&\leq \  \bigr| M^{1/2} + m^{-1/2} -2 \bigr| \, \pi_{\HR}(x,A_{x,\epsilon}\cap A_{\tilde{x}, \epsilon}^c) \nonumber \\
&\leq \  \bigr( M^{1/2} + m^{-1/2} +2 \bigr) \int_{\mathbb{S}^{d-1}} \int_{\mathbb{R}}  1_{A_{x,\epsilon}\cap A_{\tilde{x}, \epsilon}^c}(x+hv) \, \gamma^\C_{x,v}(dh) \, \sigma_{d-1}(dv) \nonumber \\
&\leq \  \bigr( M^{1/2} + m^{-1/2} +2 \bigr)  \int_{\mathbb{S}^{d-1}} \int_\epsilon^{\epsilon+|x-\tilde x|}dh\, |C^{-1/2}v| \, \sigma_{d-1}(dv) \nonumber \\
\ &\leq \  \bigr( M^{1/2} + m^{-1/2} +2 \bigr) \kappa^{1/2}\bigr|\C^{-1/2}(x-\tilde x)\bigr| \;.
\label{eq:re4a2}
\end{align}
The last two displays together give
\begin{align}
\rn{4}_a \ &\leq \ \Bigr(\frac{2}{\epsilon m^{1/2}}+\bigr( M^{1/2} + m^{-1/2} +2 \bigr) \kappa^{1/2}\Bigr)\bigr|\C^{-1/2}(x-\tilde{x})\bigr| \;. \label{eq:re4a}
\end{align}
On the other hand,  by applying \eqref{ieq:log} and \eqref{eq:Axepsic2b},
\begin{align}
\rn{4}_b \ &\leq \ \int_{A_{x,\epsilon}^c}  \left(\frac{|\C^{-1/2}(y-x)|}{|y-x|} + \frac{\left|y-\tilde{x} \right|}{\left|\C^{-1/2} (y-\tilde{x}) \right|} - 2\right) \,  \pi_{\HR}(x,dy) \nonumber\\
&\leq \ \bigr| M^{1/2} + m^{-1/2} -2 \bigr| \, \pi_{\HR}(x,A_{x,\epsilon}^c) \ \leq \  2\epsilon M^{1/2}\bigr( M^{1/2} + m^{-1/2} +2 \bigr)\;. \label{eq:re4b}
\end{align}
Inserting \eqref{eq:re4a} and \eqref{eq:re4b} into \eqref{eq:re4_split} yields \begin{equation} \label{eq:re4}
\rn{4} \ \leq \ \Bigr(\frac{2}{\epsilon m^{1/2}}+\bigr( M^{1/2} + m^{-1/2} +2 \bigr) \kappa^{1/2}\Bigr)\bigr|\C^{-1/2}(x-\tilde{x})\bigr| + 2\epsilon M^{1/2}\bigr( M^{1/2} + m^{-1/2} +2 \bigr)\;.
\end{equation}

Inserting \eqref{eq:re1}, \eqref{eq:re2}, \eqref{eq:re3} and \eqref{eq:re4} into \eqref{eq:re} yields
\begin{align*}
&\mathcal H \bigr( \pi_{\HR}(x,\cdot) \bigr|\pi_{\HR}(\tilde{x},\cdot)\bigr)
\ \leq\ \Bigr[\epsilon^{-1}m^{-1/2}|\C^{-1/2} x| + 2\epsilon M^{1/2} + 1\Bigr]\bigr|\C^{-1/2} (x-\tilde{x}) \bigr|^2\\
&\quad+\Bigr[2\epsilon^{-1}m^{-1/2}|\C^{-1/2}x|^2 + 2|\C^{-1/2}x| + \bigr(\epsilon^{-1}m^{-1/2} + 2\epsilon\kappa^{1/2}\bigr)(d-1) \\
&\qquad+ 2\epsilon^{-1}m^{-1/2}+\bigr( M^{1/2} + m^{-1/2} +2 \bigr) \kappa^{1/2}\Bigr]\bigr|\C^{-1/2} (x-\tilde{x})\bigr| \\
&\quad+ \Bigr[3M^{1/2}|\C^{-1/2} x|^2 + 2(1+\log\epsilon^{-1})M^{1/2}(d-1) + 2M^{1/2}\bigr( M^{1/2} + m^{-1/2} +2 \bigr)\Bigr]\epsilon\;.
\end{align*}

Finally, inserting the last display into Pinsker’s inequality \eqref{eq:Pinsker} and using subadditivity of the square root, we obtain 
\begin{align*}
&\TV\bigr(\pi_{\HR}(x,\cdot),\,\pi_{\HR}(\tilde{x},\cdot)\bigr)
\ \leq\ \sqrt 2\,\mathcal H \bigr( \pi_{\HR}(x,\cdot) \bigr|\pi_{\HR}(\tilde{x},\cdot)\bigr)^{1/2} \\
&\qquad \leq\ \sqrt 2\Bigr(C_1(x)^{1/2}\bigr|\C^{-1/2}(x-\tilde{x})\bigr| + C_2(x)^{1/2}\bigr|\C^{-1/2}(x-\tilde{x})\bigr|^{1/2} + C_3(x)\epsilon^{1/2}\Bigr)
\end{align*}
finishing the proof.
\end{proof}

We now extend the overlap bound of Lemma \ref{lem:overlap1} to arbitrary initial distributions.

\medskip

\begin{lemma} \label{lem:regularization}
For any $\epsilon \in (0,1)$ and $\eta, \nu \in \mathcal{P}(\mathbb{R}^d)$,
\begin{equation*}
\TV\bigr(\eta \pi_{\HR} ,\, \nu \pi_{\HR} \bigr) \ \leq \ \sqrt{2} \left(  \bigr(\eta(C_1)\bigr)^{1/2} \, \W^2_{\C^{-1/2}}(\eta, \nu) + \bigr(\eta(C_2)\bigr)^{1/2} \, \W^2_{\C^{-1/2}}(\eta, \nu)^{1/2} + \eta(C_3) \, \epsilon^{1/2}  \right)  \;.
\end{equation*}
\end{lemma}

\begin{proof}[Proof of Lemma \ref{lem:regularization}]
Let $(X,\widetilde{X})$ be an arbitrary coupling of $\eta$ and $\nu$. By the coupling characterization of the TV distance and Lemma~\ref{lem:overlap1},
\begin{align*}
&\TV\bigr( \eta \pi_{\HR}, \, \nu \pi_{\HR} \bigr)
\ \leq \ \mathbb E\TV\bigr(\pi_{\HR}(X,\cdot),\,\pi_{\HR}(\widetilde{X},\cdot)\bigr)  \\
& \qquad \leq \  \sqrt{2} \, \mathbb E\Bigr( C_1(X)^{1/2} |X - \widetilde{X}|_{\C^{-1/2}}  +  C_2(X)^{1/2} |X - \widetilde{X} |_{\C^{-1/2}}^{1/2}  + C_3(X) \epsilon^{1/2} \Bigr)\;.
\end{align*}
By Cauchy-Schwarz and Jensen's inequality,
\begin{align*}
\TV\bigr( \eta \pi_{\HR}, \, \nu \pi_{\HR} \bigr)  \ &\leq \  \sqrt{2} \Bigr(\bigr(\eta(C_1)\bigr)^{1/2}\bigr(\mathbb E |X - \widetilde{X}|_{\C^{-1/2}}^2 \bigr)^{1/2} \\
& \qquad \qquad + \bigr(\eta(C_2)\bigr)^{1/2}\bigr(\mathbb E |X - \widetilde{X}|_{\C^{-1/2}}^2 \bigr)^{1/4}  + \eta(C_3)\epsilon^{1/2} \Bigr)\;.
\end{align*}
Take the infimum over all couplings of $\eta$ and $\nu$ to finish the proof.  
\end{proof}

After proving the required overlap bounds, we are positioned to prove the mixing time upper bound for Hit-and-Run.

\begin{proof}[Proof of Theorem \ref{thm:HRmixing}]
Let $C,C'>0$ be an absolute constants that may change from line to line.

Inserting $\gamma^\C(|x|_{\C^{-1/2}}^2)=d$ and $\gamma^\C(|x|_{\C^{-1/2}})\leq d^{1/2}$ into the constants of Lemma~\ref{lem:overlap1} yields
\begin{align}
    \gamma^\C(C_1) \ &\leq \ C\max(\kappa,M,m^{-1})^{1/2}d^{1/2}\epsilon^{-1} \;, \label{eq:gammaC1}\\ 
    \gamma^\C(C_2) \ &\leq \ C\max(\kappa,M,m^{-1})d\epsilon^{-1} \;, \label{eq:gammaC2}\\
    \gamma^\C(C_3)\epsilon^{1/2} \ &\leq \ C\max(\kappa,M,m^{-1})^{1/2}d^{1/2}(1+\log\epsilon^{-1})^{1/2}\epsilon^{1/2} \;. \label{eq:gammaC3}
\end{align}

Combining Lemma~\ref{lem:HRcontr} and Lemma~\ref{lem:regularization}, we see for any  $\nu \in \mathcal{P}(\mathbb{R}^d)$ and any $n \in \mathbb{N}$
\begin{equation*}
\begin{aligned}
& \TV\bigr(\gamma^{\C}  ,\, \nu \pi_{\HR}^{n+1} \bigr) \ = \ \TV\bigr(\gamma^{\C} \pi_{\HR}^{n+1} ,\, \nu \pi_{\HR}^{n+1} \bigr)  \\
& ~~  \leq \sqrt{2} \left(  \bigr(\gamma^{\C}(C_1)\bigr)^{1/2} \,  \W^2_{\C^{-1/2}}(\gamma^{\C}, \nu) e^{-\rho n} + \bigr(\gamma^{\C}(C_2)\bigr)^{1/2}  \W^2_{\C^{-1/2}}(\gamma^{\C}, \nu)^{1/2} e^{-\rho n/2} + \gamma^{\C}(C_3) \epsilon^{1/2}  \right)\;.
\end{aligned}
\end{equation*}
The right hand side is bounded above by $\eps>0$ if the following three bounds hold:
\begin{align}
    \bigr(\gamma^{\C}(C_1)\bigr)^{1/2} \,  \W^2_{\C^{-1/2}}(\gamma^{\C}, \nu) e^{-\rho n} \ &\leq\ \frac\eps{3\sqrt 2}\;, \label{eq:ncond1}\\
    \bigr(\gamma^{\C}(C_2)\bigr)^{1/2}  \W^2_{\C^{-1/2}}(\gamma^{\C}, \nu)^{1/2} e^{-\rho n/2} \ &\leq \ \frac\eps{3\sqrt 2}\;, \label{eq:ncond2}\\ 
    \gamma^{\C}(C_3) \epsilon^{1/2} \ &\leq \ \frac\eps{3\sqrt 2}\;. \label{eq:epscond}
\end{align}
Using $(1+\log\epsilon^{-1})\epsilon\leq2\epsilon^{1/2}$ in \eqref{eq:gammaC3}, we see that \eqref{eq:epscond} holds for
\[ \epsilon\ =\ C\max(\kappa,M,m^{-1})^{-2}d^{-2}\eps^4\quad\text{with a suitable absolute constant $C$.} \]
Inserting this choice into \eqref{eq:gammaC1} and \eqref{eq:gammaC2} yields
\[ \max\bigr(\gamma^\C(C_1),\,\gamma^\C(C_2)\bigr)\ \leq\ C\max(\kappa,M,m^{-1})^3d^3\eps^{-4}\;. \]
As conditions \eqref{eq:ncond1} and \eqref{eq:ncond2} are guaranteed for
\[ n\ \geq\ C\rho^{-1}\log\Bigr(C'\max\bigr(\gamma^\C(C_1),\,\gamma^\C(C_2)\bigr)\,\eps^{-1}\,\W^2_{\C^{-1/2}}(\gamma^{\C}, \nu)\Bigr) \]
with suitable absolute constants $C,C'$, the mixing time satisfies
\[ \tmix(\eps,\nu)\ \leq\ C\rho^{-1}\log\Bigr(C'\max(\kappa,M,m^{-1})\,d\,\eps^{-1}\W^2_{\C^{-1/2}}(\gamma^{\C}, \nu)\Bigr)\;. \]
\end{proof}

\section{Coordinate-free randomized Kaczmarz algorithm}\label{sec:kaczmarz}

The Kaczmarz algorithm \cite{kaczmarz,gower15,gower15b} is an iterative method for approximately solving overdetermined linear systems of the form
\begin{equation}\label{eq:linsystem}
    Ax\ =\ b
\end{equation}
where $A\in\mathbb R^{d\times m}$ is a full rank matrix with $d\geq m$, and $b\in\mathbb R^d$.  The algorithm leverages the geometric interpretation of the solution to \eqref{eq:linsystem} as the intersection of the hyperplanes
\begin{equation}\label{eq:Hi}
    H_i\ =\ \bigr\{x\in\mathbb R^m:e_i\cdot(Ax-b)=0\bigr\}\quad\text{for $1\leq i\leq d$}
\end{equation}
where $e_i$ is the $i$-th canonical basis vector of $\mathbb R^d$.  

Starting from an initial state $x_0\in\mathbb R^m$, the Kaczmarz method iteratively updates the solution. At each iteration, a basis vector $e_i$ is selected, and the next state $x_{k+1}$ is computed as the orthogonal projection of the current state $x_k$ onto the corresponding hyperplane $H_i$.  The classical Kaczmarz method selects the basis vectors deterministically in a cyclic order \cite{kaczmarz}.  

Strohmer and Vershynin \cite{vershynin09} later introduced and analyzed a randomized variant of the Kaczmarz method, where the basis vectors are selected randomly, with probabilities proportional to $|A^Te_i|^2$.  They provided a sharp bound on the mean-squared convergence  to the solution of \eqref{eq:linsystem}. 

In each iteration, the direction of update in the randomized Kaczmarz method is restricted to $\operatorname{span}(A^Te_i)$, i.e., normal to the selected hyperplane $H_i$.  Analogously, the random-scan Gibbs sampler is restricted to move in $\operatorname{span}(e_i)$, corresponding to a randomly selected basis vector.  The possible improvement from diffusive to ballistic mixing in Hit-and-Run --- achieved by allowing random-scan Gibbs to move in any coordinate-free, randomly sampled direction $v\in\mathbb S^{d-1}$ --- motivates us to consider a similar coordinate-free extension for the randomized Kaczmarz algorithm.

We begin by generalizing \eqref{eq:Hi} from coordinate-bound unit vectors to coordinate-free unit vectors.  Specifically, the solution to \eqref{eq:linsystem} can be expressed as the intersection of the hyperplanes 
\[ H(v)\ =\ \bigr\{x\in\mathbb R^m:v\cdot(Ax-b)=0\bigr\}\quad\text{for $v\in\mathbb S^{d-1}$.} \]
Let $\tau$ be a probability measure on $\mathbb S^{d-1}$.   The iterations of the \emph{generalized randomized Kaczmarz algorithm} are defined as follows: select $v\sim\tau$, then compute the updated state $x_{k+1}$ by projecting the current state $x_k$ orthogonally onto $H(v)$.  Since
\[ H(v)\ =\ \operatorname{span}\Bigr(A^Tv\Bigr)^\perp+\frac{v\cdot b}{|A^Tv|}\frac{A^Tv}{|A^Tv|}\;, \]
the update step can be written explicitly as
\[ x_{k+1}\ =\ \Pi_{A^Tv}\,x_k+\frac{v\cdot b}{|A^Tv|}\frac{A^Tv}{|A^Tv|}\quad\text{for $v\sim\tau$.} \]
We again have a special interest in the case $\tau=\Unif(\mathbb S^{d-1})$, which corresponds to the \emph{coordinate-free randomized Kaczmarz algorithm}.

\begin{figure}[t]
\centering
\noindent\begin{minipage}[b]{.3\textwidth}
\centering
\begin{tikzpicture}[scale=1.6]
%\draw (-1.5,-1.5) rectangle (1.5,1.5);
\draw (-1.5,0) -- (1.5,0);
\draw (0,-1.5) -- (0,1.5);
\draw[line width=1pt] (0,0) circle (1);
\foreach \k in {0,...,23} {\filldraw[gray] ({cos((\k/24)*2*pi r)},{sin((\k/24)*2*pi r)}) circle (2pt);}
\foreach \k in {0,6,12,18} {\filldraw ({cos((\k/24)*2*pi r)},{sin((\k/24)*2*pi r)}) circle (2pt);}
\end{tikzpicture}
\caption*{$\law(v)=\Unif(\mathbb S^1)$}
\end{minipage}%
\noindent\begin{minipage}[b]{.3\textwidth}
\centering
\begin{tikzpicture}[scale=1.6]
%\draw (-1.5,-1.5) rectangle (1.5,1.5);
\draw (-1.5,0) -- (1.5,0);
\draw (0,-1.5) -- (0,1.5);
\draw [variable=\t,line width=1pt,smooth,samples=100,domain=0:2*pi] plot({0.2*sin(\t r)},{cos(\t r)+sin(\t r)});
\foreach \k in {0,...,23} {\filldraw[gray] ({0.2*sin((\k/24)*2*pi r)},{cos((\k/24)*2*pi r)+sin((\k/24)*2*pi r)}) circle (2pt);}
\foreach \k in {0,6,12,18} {\filldraw ({0.2*sin((\k/24)*2*pi r)},{cos((\k/24)*2*pi r)+sin((\k/24)*2*pi r)}) circle (2pt);}
\end{tikzpicture}
\caption*{$\law(A^Tv)$}
\end{minipage}%
\noindent\begin{minipage}[b]{.3\textwidth}
\centering
\begin{tikzpicture}[scale=1.6]
%\draw (-1.5,-1.5) rectangle (1.5,1.5);
\draw (-1.5,0) -- (1.5,0);
\draw (0,-1.5) -- (0,1.5);
\draw[line width=1pt] (0,0) circle (1);
\foreach \k in {0,...,23} {\filldraw[gray] ({0.2*sin((\k/24)*2*pi r)/sqrt(0.2^2*sin((\k/24)*2*pi r)^2+(cos((\k/24)*2*pi r)+sin((\k/24)*2*pi r))^2)},{(cos((\k/24)*2*pi r)+sin((\k/24)*2*pi r))/sqrt(0.2^2*sin((\k/24)*2*pi r)^2+(cos((\k/24)*2*pi r)+sin((\k/24)*2*pi r))^2)}) circle (2pt);}
\foreach \k in {0,6,12,18} {\filldraw ({0.2*sin((\k/24)*2*pi r)/sqrt(0.2^2*sin((\k/24)*2*pi r)^2+(cos((\k/24)*2*pi r)+sin((\k/24)*2*pi r))^2)},{(cos((\k/24)*2*pi r)+sin((\k/24)*2*pi r))/sqrt(0.2^2*sin((\k/24)*2*pi r)^2+(cos((\k/24)*2*pi r)+sin((\k/24)*2*pi r))^2)}) circle (2pt);}
\end{tikzpicture}
\caption*{$\law(A^Tv/|A^Tv|)$}
\end{minipage}%
\caption{\textit{Distribution of the directions $A^Tv/|A^Tv|$, $v\sim\Unif(\mathbb S^1)$, along which coordinate-free randomized Kaczmarz moves, for $A$ as in \eqref{eq:A}.  The gray beads illustrate the redistribution of probability mass, favoring directions nearly orthogonal to $e_1$ (resulting in local moves along $e_1$) while retaining enough mass in directions allowing global moves along $e_1$.  This enables the coordinate-free method to converge at the ballistic rate in \eqref{eq:kaczmarz_ballistic}.  In contrast, the black beads represent the fixed directions to which randomized Kaczmarz is restricted, explaining its slower, diffusive rate in \eqref{eq:kaczmarz_diffusive}.}}
\label{fig:kaczmarzlaw}
\end{figure}

\medskip

\begin{lemma}\label{lem:kaczmarz}
Let $x^\ast\in\mathbb R^m$ be the solution to \eqref{eq:linsystem}.
Starting from any initial state $x_0\in\mathbb R^m$, the iterates $(x_k)_{k\geq0}$ of the generalized randomized Kaczmarz algorithm converge to $x^\ast$ in the sense that
\begin{equation}  \label{eq:rateKaczmarz}
\sqrt{\mathbb E|x_k-x^\ast|^2}\ \leq\ (1-\rho)^k|x_0-x^\ast| \quad \text{with rate} \quad 
    \rho\ =\ \frac12\inf_{|\zeta|=1}\mathbb E_{v\sim\tau}\Bigr(\zeta\cdot\frac{A^Tv}{|A^Tv|}\Bigr)^2\;.
\end{equation}
\end{lemma}

In the special case where the probability measure $\tau$ is given by
\[ \tau\ =\ \|A\|_F^{-2}\sum_{i=1}^d|A^Te_i|^2\delta_{e_i} \]
the \emph{generalized randomized Kaczmarz algorithm}  reduces to the classical \emph{randomized Kaczmarz algorithm}.  In this case, Lemma \ref{lem:kaczmarz} recovers the sharp convergence rate from \cite{vershynin09}:
\begin{equation}\label{eq:raterandKaczmarz}
    \rho\ =\ \frac12\|A\|_F^{-2}\inf_{|\zeta|=1}\sum_{i=1}^d|A^Te_i|^2\Bigr(\zeta\cdot\frac{A^Te_i}{|A^Te_i|}\Bigr)^2\ =\ \frac12\|A\|_F^{-2}\inf_{|\zeta|=1}|A\zeta|^2\ \geq\ \frac12\|A\|_F^{-2}\|A^{-1}\|^{-2} \;.
\end{equation}
Here, $\| \cdot \|_F$ denotes the Frobenius matrix norm, and $\| \cdot \|$ denotes the operator norm.  The inequality in \eqref{eq:raterandKaczmarz} follows from the fact that the left inverse $A^{-1}$ satisfies $|z|\leq\|A^{-1}\||Az|$ for all $z \in \mathbb{R}^m$.

\begin{proof}[Proof of Lemma~\ref{lem:kaczmarz}]
Let $x_0\in\mathbb R^m$ and consider synchronously coupled iterations starting from $x_0$ and $x^\ast$:
\[ x_1\ =\ \Pi_{A^Tv}\,x_0+\frac{v\cdot b}{|A^Tv|}\frac{A^Tv}{|A^Tv|}\quad\text{and}\quad x^\ast\ =\ \Pi_{A^Tv}\,x^\ast+\frac{v\cdot b}{|A^Tv|}\frac{A^Tv}{|A^Tv|}\quad\text{using the same $v\sim\tau$.} \]
Since $x^\ast$ is a fixed point of the iteration,
\[ x_1-x^\ast\ =\ \Pi_{A^Tv}\,(x_0-x^\ast) \;. \]
By Lemma \ref{lem:contrproj}, we have
\[ \mathbb E_{v\sim\tau}|x_1-x^\ast|^2\ =\ \mathbb E_{v\sim\tau}\bigr|\Pi_{A^Tv}\,(x_0-x^\ast)\bigr|^2\ \leq\ (1-2\rho)|x_0-x^\ast|^2 \]
where $\rho$ is defined in \eqref{eq:rateKaczmarz}.  The proof is completed by iterating this bound, taking square roots, and using $\sqrt{1 - 2 \mathsf{x}} \le 1- \mathsf{x}$ valid for $\mathsf{x} \in [0,1/2)$.
\end{proof}

\begin{figure}[t]
\noindent\begin{minipage}[b]{0.25\textwidth}
\includegraphics[width=\textwidth]{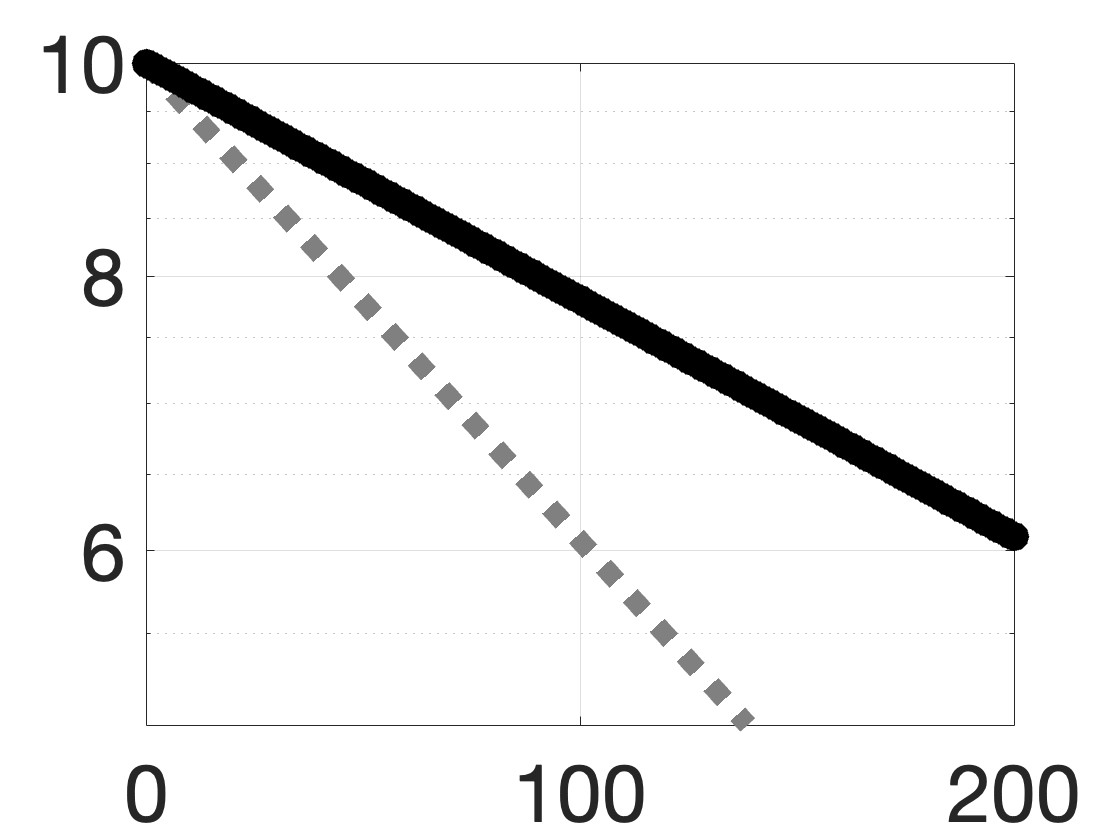}
\caption*{(a) $a=0.1$}
\end{minipage}%
\noindent\begin{minipage}[b]{0.25\textwidth}
\includegraphics[width=\textwidth]{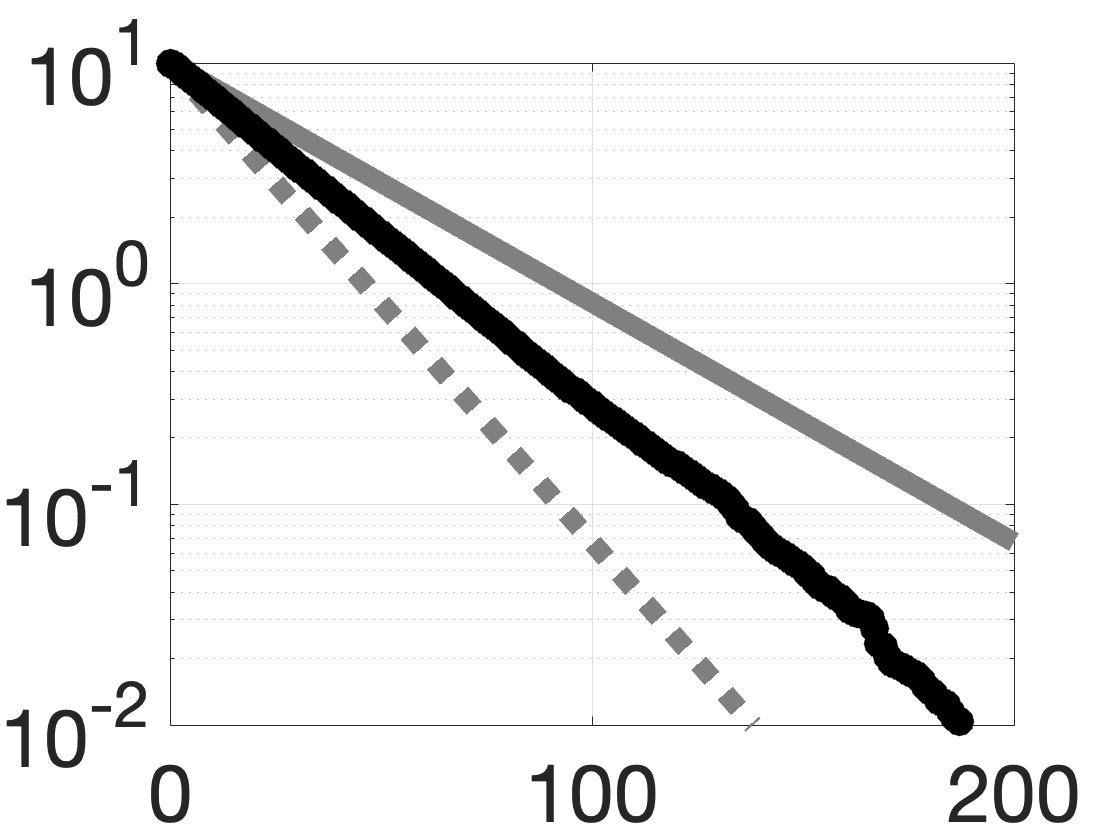}
\caption*{(b) $a=0.1$}
\end{minipage}%
\noindent\begin{minipage}[b]{0.25\textwidth}
\includegraphics[width=\textwidth]{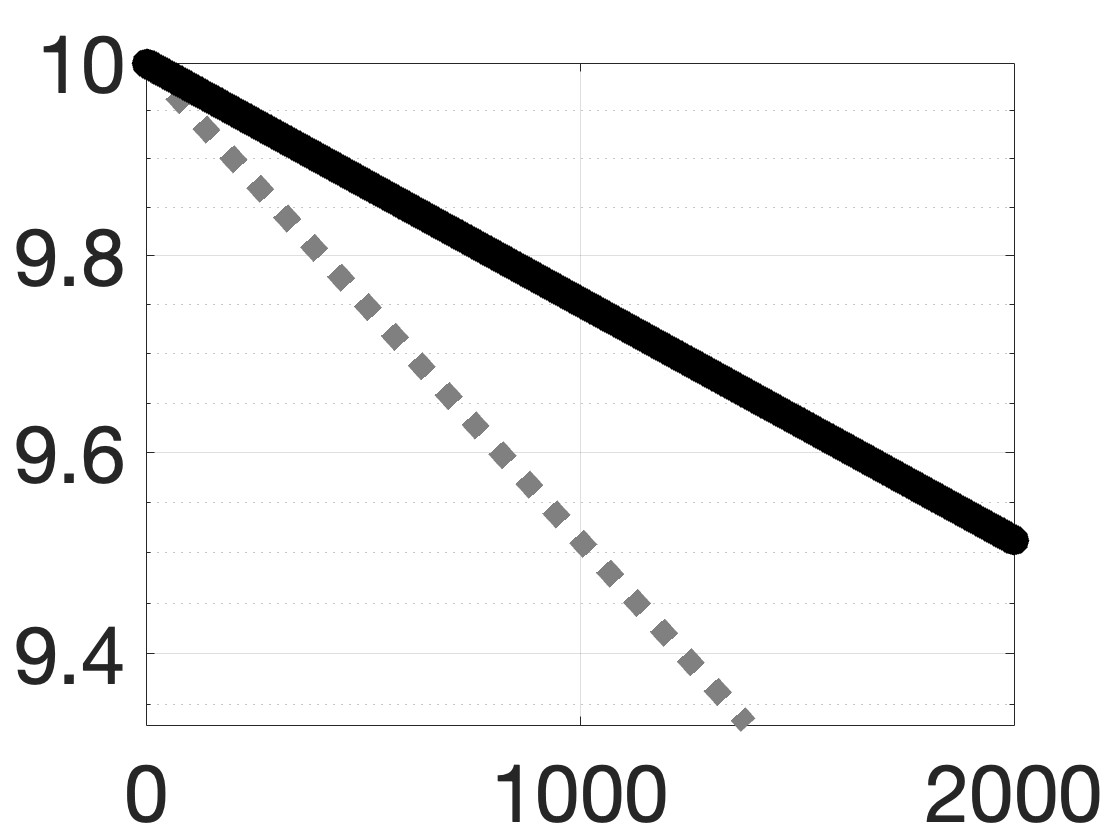}
\caption*{(c) $a=0.01$}
\end{minipage}%
\noindent\begin{minipage}[b]{0.25\textwidth}
\includegraphics[width=\textwidth]{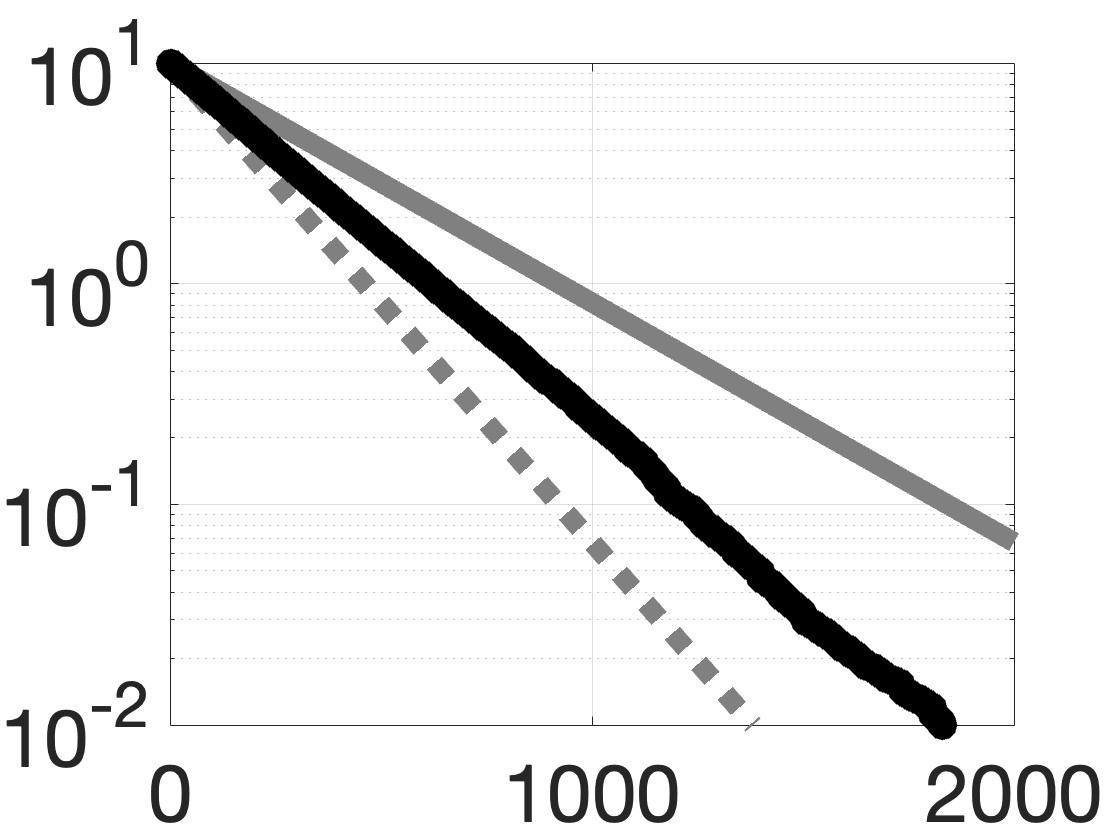}
\caption*{(d) $a=0.01$}
\end{minipage}%
\caption{\textit{Convergence of randomized Kaczmarz (a), (c) vs. coordinate-free randomized Kaczmarz (b), (d) for $Ax=0$, with $A$ as defined in \eqref{eq:A}, starting from $x_0=(-10,0)$.  Solid black lines show the mean error $\sqrt{\mathbb E|x_k-x^\ast|^2}$ averaged over 10,000 realizations.  Only the coordinate-free variant converges within the displayed number of iterations.  Solid gray lines show the theoretically guaranteed diffusive rate $a^2/4$ in (a),(c) and ballistic rate $a/4$ in (b), (d), while dotted gray lines indicate twice these theoretical rates.  Randomized Kaczmarz converges exactly at the theoretical rate, while coordinate-free randomized Kaczmarz converges within an absolute constant of the theoretical rate.
}}
\label{fig:kaczmarz_sim}
\end{figure}

In certain settings, Lemma \ref{lem:kaczmarz} implies a diffusive to ballistic speed-up for \emph{coordinate-free randomized Kaczmarz} compared to the classical \emph{randomized Kaczmarz}, analogous to the speed-up observed for Hit-and-Run compared to random-scan Gibbs.  To illustrate, consider the bivariate example
\begin{equation}\label{eq:A}
    A\ =\ \begin{pmatrix} 0 & 1 \\ a & 1 \end{pmatrix}\quad\text{for $a\in(0,1)$ small.}
\end{equation}
Randomized Kaczmarz, being restricted to move along the directions $(0,1)$ and $(a,1)$, is constrained to directions that are nearly orthogonal to $e_1$.  As a result, it exhibits slow convergence due to predominantly local moves along $e_2$.  In contrast, the coordinate-free randomized Kaczmarz algorithm moves in the direction $A^Tv$ with $v\sim\Unif(\mathbb S^1)$, enabling global moves along $e_1$ with sufficiently high probability, thereby accelerating convergence, see Figure \ref{fig:kaczmarzlaw}.

For randomized Kaczmarz, the diffusive rate can be derived from \eqref{eq:raterandKaczmarz}:
\begin{equation}\label{eq:kaczmarz_diffusive}
\rho\ =\ \frac12\|A\|_{F}^{-2}\inf_{|\zeta|=1}|A\zeta|^2\ =\ \frac{a^2}{2(2+a^2)}\ \asymp\ \frac{a^2}{4}\;.
\end{equation}
On the other hand, for the coordinate-free randomized Kaczmarz algorithm, the rate of convergence improves to the ballistic rate:
\begin{equation}\label{eq:kaczmarz_ballistic}
\rho\ =\ \frac12\inf_{|\zeta|=1}\mathbb E_{v\sim\Unif(\mathbb S^1)}\Bigr(\zeta\cdot\frac{A^Tv}{|A^Tv|}\Bigr)^2\ =\ \frac{a(1+a)}{2(2+a(2+a))}\ \asymp\ \frac{a}{4}\;.
\end{equation}
This diffusive to ballistic speed-up is numerically illustrated in Figure \ref{fig:kaczmarz_sim}.

\printbibliography

\end{document}